\documentclass{amsart}
\usepackage{amsmath}
\usepackage{amssymb}
  \usepackage{paralist}
  \usepackage{graphics}
  \usepackage{epsfig} 
 \usepackage[colorlinks=true]{hyperref}
   
\newtheorem{theorem}{Theorem}[section]

\newtheorem{lemma}[theorem]{Lemma}
\newtheorem{proposition}{Proposition}

\theoremstyle{definition}

\newtheorem{remark}{Remark}

\renewcommand{\geq}{\geqslant}
\renewcommand{\leq}{\leqslant}

\newcommand{\cadlag}{c\`adl\'ag}
\title[Discrete limit and monotonicity properties]
      {Discrete limit and monotonicity properties of the  Floquet eigenvalue in an age structured  cell division cycle model}

\author[St\'ephane Gaubert and Thomas Lepoutre]{}

\subjclass[2010]{35B40, 35Q92, 35P15}

 \keywords{cell cycle, circadian rhythms, structured PDEs, Perron-Frobenius theory, Floquet eigenvalue, delay differential equations.}
 \email{stephane.gaubert@inria.fr}
 \email{thomas.lepoutre@inria.fr}

\begin{document}
\maketitle
\centerline{\scshape St\'ephane Gaubert }
\medskip
{\footnotesize
\centerline{INRIA Saclay \^Ile-de-France and CMAP, \'Ecole Polytechnique}
\centerline{Postal address: CMAP, \'Ecole Polytechnique, 91128 Palaiseau cedex, France}} 

\medskip

\centerline{\scshape Thomas  Lepoutre}
\medskip
{\footnotesize
 \centerline{INRIA Rh\^one Alpes (team DRACULA)}
\centerline{
Batiment CEI-1,
66 Boulevard NIELS BOHR, 69603 Villeurbanne cedex, France }
      \centerline{Universit\'e de Lyon
CNRS UMR 5208}
      \centerline{
Universit\'e Lyon 1
Institut Camille Jordan}
      \centerline{
43 blvd. du 11 novembre 1918
F-69622 Villeurbanne cedex
France}
}

\bigskip

%
\begin{abstract}
We consider a cell population described by an age-structured 
partial differential equation with time periodic coefficients.
We assume that 
division only occurs after a minimal
age (majority) and within certain time intervals. 
We study the asymptotic behavior of the dominant
Floquet eigenvalue, or Perron-Frobenius eigenvalue, 
representing the growth rate, as a function of the
majority age, when the division rate tends to infinity (divisions
become instantaneous). We show
that the dominant Floquet eigenvalue converges to a staircase
function with an infinite number of steps,
determined by a discrete dynamical system. 
As an intermediate result, we give a structural condition
which guarantees that 
the dominant Floquet eigenvalue is a nondecreasing function
of the division rate. We also give a counter example
showing that the latter monotonicity property does
not hold in general.

\end{abstract}
\newcommand{\supp}{\operatorname{supp}}
\newcommand{\scaling}{\kappa}
\newcommand{\floor}[1]{\lfloor #1\rfloor}
\newcommand{\LL}{\mathcal{L}}
\newcommand{\M}{\mathcal{M}}
\newcommand{\N}{\mathcal{N}}
\newcommand {\myproof} {\noindent {\bf Proof}. }
\newcommand{\fer}[1]{(\ref{#1})}
\newcommand {\f}   {\frac}
\newcommand {\p}   {\partial}
\newcommand{\R}{\mathbb{R}}
\def\myexp#1{\exp\big(#1\big)}
\newcommand{\dis}{\displaystyle}

\newcommand{\beq}{\begin{equation}}
\newcommand{\eeq}{\end{equation}}
\newcommand{\bea} {\begin{array}{rl}}
\newcommand{\eea} {\end{array}}
\newcommand{\bepa}{\left\{ \begin{array}{l}}
\newcommand{\eepa} {\end{array}\right.}
\def\bbbone{{\mathchoice {\rm 1\mskip-4mu l} {\rm 1\mskip-4mu l}{\rm 1\mskip-4.5mu l} {\rm 1\mskip-5mu l}}}

\vspace*{1cm}



%
%
%
%
%
%
\section{Introduction}
\subsection{Age structured model of the cell division cycle}
Among all the population  models, age-structured population models are perhaps the simplest and 

oldest ones taking into account the variability between individuals. Age structured partial differential equations have been introduced in epidemiology through the famous Von Foerster - Mac Kendrick model~\cite{McKendrick} (where age stands for age in the disease). In its simplest version, it describes the dynamics of a population represented by a density $n(t,x)$ of individuals of age between $x$ and $x+dx$ through age-dependent birth and death rate:

\begin{equation}
\left\lbrace\begin{array}{l}\label{eq:renewal}
\dfrac{\partial n}{\partial t}+\dfrac{\partial n}{\partial x} +d(x)n(t,x) =0,\\[0.3cm]
n(t,0)=\displaystyle\int_0^\infty B(x) n(t,x) dx,\\[0.3cm]
n(0,x)=n^0(x) \quad given. 
\end{array}\right.
\end{equation}

A special case of this model can be used to represent cell division, taking $d(x)=K(x)$ and $B(x)=2K(x)$. 
More complex models, in terms of structure, representing the cell division cycle (linear or not, size or  cyclin  structured) have been developped and studied in \cite{MetzDiekmann,Doumic,Webb_book}.  
In this paper, we study an extension of the age-structured division equation, in which the coefficients are periodic functions of time. There are several motivations to introduce time dependence in \eqref{eq:renewal}. The most classical one is the representation of seasonality  (see \cite{Bacaer} for instance). Concerning cell division, our  main motivation comes from cancer therapy such as resonance therapy \cite{Dibrov} or chronotherapy \cite{Levi}. The latter is based on the use of circadian 
rhythms and motivates especially the introduction of time periodic coefficients (see \cite{CMP2, Goldbeterbook} for instance). We also refer to \cite{Diekmann_periodic,Magal_periodic,Thieme_84}  for work on renewal models with periodic coefficients.

In our case, the periodic division equation reads 
\begin{equation}
\left\lbrace\begin{array}{l}\label{eq:periodic_division}
\dfrac{\partial n}{\partial t}+\dfrac{\partial n}{\partial x} +K(t,x)n(t,x) =0,\\[0.3cm]
n(t,0)=2\displaystyle\int_0^\infty K(t,x) n(t,x) dx,\\[0.3cm]
n(0,x)=n^0(x) \quad given. 
\end{array}\right.
\end{equation}
where $K$ satisfies $K(t+T,\cdot)=K(t,\cdot)$ for some $T>0$.

For equations \eqref{eq:renewal},\eqref{eq:periodic_division}, the population grows exponentially with time (see \cite{PerthameBook} for instance) and the growth rate is measured by the so called Malthus parameter $\lambda$ ($n(t,x)\sim \rho e^{\lambda t}N(t,x)$, $N$ being given by \eqref{eq:eigendirect}). Particularly, in \eqref{eq:periodic_division}, the Malthus parameter $\lambda_F $ (F for Floquet) can be generally described by means of a Floquet eigenvalue problem.
\begin{equation}
\left\lbrace\begin{array}{l}\label{eq:eigendirect}
\dfrac{\partial N}{\partial t}+\dfrac{\partial N}{\partial x} +(\lambda_F+K(t,x))N(t,x) =0,\\[0.3cm]
n(t,0)=2\displaystyle\int_0^\infty K(t,x) N(t,x) dx,\\[0.3cm]
N(t+T,x)=N(t,x),\; N>0. 
\end{array}\right.
\end{equation}
together with the dual eigenproblem
\begin{equation}
\left\lbrace\begin{array}{l}\label{eq:eigendual}
-\dfrac{\partial \phi}{\partial t}-\dfrac{\partial \phi}{\partial x} +(\lambda_F+K(t,x))\phi(t,x) =2K(t,x)\phi(t,0),\\[0.3cm]
\phi(t+T,x)=\phi(t,x),\; \phi>0. 
\end{array}\right.
\end{equation}
For uniqueness issues, these equations are usually completed with the following normalization:
\beq\label{eq:normalisation}
\frac{1}{T}\int_0^T\int_0^\infty N dx dt =\frac{1}{T}\int_0^T\int_0^\infty N(t,x)\phi(t,x)dxdt=1. 
\eeq
Actually, the integral $\int_0^\infty N(t,x)\phi(t,x)dx$ does not depend on $t$. 
Various properties of the Floquet eigenvalue have been studied in \cite{Bacaer,Lepoutre_mmnp,Lepoutre_MCM,CGP,CMP} with in mind the comparison with time independent systems (in which $K(t,x)$  is replaced by  suitable time averages). Here, we are interested in properties of the eigenvalue (and more generally of the dynamics) that are intrinsically related to the fact that the coefficients are time-dependent. 

We consider in Section~\ref{sec:asymp} a division rate of the form
\begin{align}
\label{e-scaling} K(t,x)=\scaling\psi(t)B(x)\bbbone_{[a,+\infty[}(x)
\enspace .
\end{align}
Here, $\scaling>0$ is a scaling parameter and the function $\psi(t)\geq 0$ expresses the dependence of the division rate as a function of time.
We denote by $\bbbone$ the indicator function of a set. The term
$B(x)\bbbone_{[a,+\infty[}(x)$ indicates that the cells only divide
when $x\geq a$, i.e., when they are older than a ``majority age''
$a$. The function $B(x)$ expresses a general modulation of the
division rate as a function of the age. 

The function $B$ will be required
to be bounded, positive, of infinite integral,
\begin{equation}\label{as:B}
B\in L^\infty (\mathbb{R}^+),\quad B>0,\quad \int_{\mathbb{R}^+} B=+\infty
\end{equation}
This is
needed for  the existence theory of the Floquet eigenvalue (see the appendix).

\subsection{Main results}
Our goal is to derive the asymptotics 
of the growth rate when the division becomes instantaneous, meaning
that the scaling parameter $\scaling$ in~\eqref{e-scaling} tends to $\infty$. 
The growth rate is represented by the Floquet eigenvalue, 
$\lambda_F$, which is the unique scalar $\lambda$ solution of~\eqref{eq:eigendirect},\eqref{eq:eigendual}. Hence, we shall investigate the asymptotic behavior of $\lambda_F$ as a function of $\scaling$.

To this end, we first study the monotonicity of the Floquet eigenvalue with respect to the division coefficient $K(t,x)$. Since every division replaces
one old cell by two new ones, it is natural to think that the growth rate
is a nondecreasing function of $K(t,x)$. Our first
result is an explicit counter example, showing that this apparently intuitive property is not valid due to the time and age dependence of the division rate.
\begin{proposition}\label{theo-cex}
There exists configurations such that
$$\forall (t,x),\qquad 0\leq K^1(t,x)\leq K^2(t,x),$$
and the corresponding Floquet eigenvalues satisfy 
$$
\lambda_F^1>\lambda_F^2.
$$
\end{proposition}
However, the following theorem identifies a subclass of transition
rates for which the Floquet eigenvalue is indeed a monotone function
of the transition rate.
\begin{theorem}[Comparison principle]\label{th-comparison}
If $K^1$ satisfies 
$$
v\mapsto \int_v^tK^1(s,s-v)ds \qquad\text{is nondecreasing for any }t,
$$
then for all $K^2\geq 0$,
\begin{eqnarray*}
\left(\forall(t,x)\quad K^2(t,x)\geq K^1(t,x)\right)&\Rightarrow &\lambda_F^2\geq \lambda_F^1,\\
\left(\forall(t,x) \quad K^2(t,x)\leq K^1(t,x)\right)&\Rightarrow &\lambda_F^2\leq \lambda_F^1.
\end{eqnarray*}
\end{theorem}
The proof relies on a disaggregation idea of independent interest.
We first show that the density $n(t,x)$
can be written as $\sum_{i\geq 0} 2^i n_i(t,x)$, where every $n_i(t,x)$ 
can be interpreted as the density of an individual of the $i$th generation, and
the densities $(n_i)_{i\geq 0}$ satisfy an infinite triangular system of transport equations coupled by integral terms. Then, for all $j\geq 0$, we
show that the term $S_j(t)=\sum_{i\geq j} \int_0^\infty n_i(t,x)dx$ is a nondecreasing function of $K\in \{K^1,K^2\}$, and derive from this fact the comparison 
of Floquet eigenvalues. We note that the latter fact 
has a probabilistic interpretation: it means precisely that 
the generation of a single
individual is a nondecreasing function of $K\in \{K^1,K^2\}$
with respect to the stochastic (majorization) order.

When $\inf \psi>0$, the intuition predicts that the growth rate goes to the limit $\log 2/a$ as $\scaling$ tends to infinity, 
because every cell will divide shortly after reaching
age $a$ when the parameter $\scaling$ is large. Our next result 
confirms that this is the case.
\begin{proposition}\label{Th:log2sura}
Suppose that $\psi$ is a bounded $T$-periodic function with $\inf\psi>0$,
and that $B$ is a positive bounded function satisfying~\eqref{as:B}.

Then, for all $\scaling>0$, 
\begin{eqnarray}
0&\leq\lambda_F(\scaling)\leq &\dfrac{\log(2)}{a},\label{eq:encadrement}
\end{eqnarray}
and
\begin{eqnarray}
\lim_{\scaling\to\infty}\lambda_F(\scaling)&=
&\frac{\log(2)}{a}\label{eq:limite}.
\end{eqnarray}
\end{proposition}

When $\psi$ vanishes, meaning that the division of cells is blocked at certain
times of the day, we shall see that complex synchronization phenomena 
appear. We shall assume that $\psi$ has a square wave shape, i.e., that $\psi$ is a $T$-periodic
function such that, for some $0<\tau< T$,
\begin{equation}\label{as:tauT}
\psi(t)=
\begin{cases}
1 & \text{for }t\in [0,\tau)\\
0 & \text{for }t\in [\tau,T)\\
\end{cases}
.\end{equation}
 
\newcommand{\residu}{b}
The following is the main result of this paper. We denote by $\lceil x\rceil$ the smallest integer which is not less than a real number $x$.
\begin{theorem}[Discrete limit]\label{Th:devil}
Suppose that the division rate is given by 
$K(t,x)=\scaling\psi(t)\bbbone_{[a,\infty)}(x)$,
where $\psi$ is given by (\ref{as:tauT}). Denote $W_\tau=[\tau,T]+\mathbb{N}T$ and let $N_a$ be the integer defined by
\[
N_a a=\min W_\tau\cap \mathbb{N}a \enspace. 
\]
Then, the limit of the Floquet eigenvalue $\lambda^\scaling(a)$ as $\scaling\rightarrow +\infty$ is given by 
\begin{equation}\label{eq:lambda_devil}
\lambda^\infty(a)=\frac{N_a}{\lceil N_a a/T\rceil T}\log 2.
\end{equation} 
Moreover, we have the estimate 
\begin{equation}\label{eq:devil_bounds}
\lambda^\infty(a,\tau)\geq\lambda^\scaling(a,\tau)\geq \lambda^\infty(a)\bigg(1-e^{-\scaling r(a,\tau)}\bigg),
\end{equation}
where the function $r(a,\tau)$ (the convergence rate) satisfies
\beq
r(a,\tau)\geq \min\left(\frac{a_r-a}{2},\tau\right), \quad a_r=\sup\{a',\lambda^\infty(a')=\lambda^\infty(a)\},
\eeq
and $r$ is positive outside a closed countable set.
\end{theorem} 
The function $a\mapsto \lambda^\infty(a)$ is a staircase with a countable number of steps (see Figure~\ref{fig:figure_devil_2tiers} below and  Figure~\ref{fig:tau05v13510} and Figure~\ref{fig:figure_devil033} in Section~\ref{sec:num}). The quantity $a_r-a$ which controls the rate represents the distance to the right end of the step containing $a$ (on Figure~\ref{fig:figure_devil_2tiers} for instance, we can observe that the speed of convergence depends on this distance). 

\begin{figure}[htbp]
	\centering
		\includegraphics[width=0.50\textwidth]{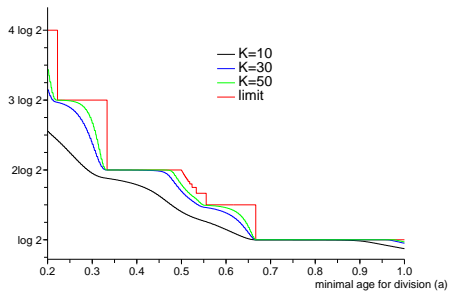}
	\caption{Convergence for $\tau=2/3$ and $T=1$. }
	\label{fig:figure_devil_2tiers}
\end{figure}
There is a special case in which the limit of the Floquet eigenvalue can be
rewritten in terms of integer parts. If $\tau/T<1/2$, and $0<a<T$, then, we have\[ \lambda^\infty(a)=\lceil\frac{\tau}{a}\rceil\log 2 \enspace .
\]
We shall derive intuitively Formula (\ref{eq:lambda_devil}) in Section~\ref{sec:derivation}, by introducing a discrete dynamical system, equipped with a multiplicative functional, the asymptotic geometric mean of which yields 
$\lambda^\infty(a)$. The existence of such a simple asymptotic formula 
reflects the fact that the Perron-Frobenius operator describing
the evolution of the population degenerates as $\scaling\to\infty$. The rate $r(a)$ appears naturally in the proof of convergence. 
This proof relies on Theorem~\ref{th-comparison} (the comparison principle), on the generational model already used in the proof of this theorem, and on 
explicit bounds of the density, for a special initial condition leading
to a trajectory ``close'' to that of the discrete dynamical system. 

The  convergence rate in Theorem~\ref{Th:devil} is reminiscent
of large-deviation or Laplace asymptotics. However, this theorem
differs in its essence of already known asymptotics results in Perron-Frobenius
theory with such large deviations flavor, including 
zero-temperature asymptotics of Ising type models or Frenkel-Kantorova
models~\cite{ABG96,thieullen}. In these problems,
the limit of the Perron eigenvalue, in a log-scale,
turns out to be the ergodic constant
of a controlled deterministic dynamical system,
whereas in the present case, the dynamical system
which determines the limit involves no control.

When $\psi$ vanishes, we may consider $\psi^\varepsilon(t):=\psi(t)+\varepsilon$,
with $\varepsilon>0$, and define the growth rate $\lambda(\epsilon, \scaling)$ with an obvious notation. A comparison of the previous theorems shows that the
limits in $\scaling$ and $\varepsilon$ do not commute
\[
\frac{\log (2)}{a} = \lim_{\varepsilon\to 0}\lim_{\scaling\to\infty} \lambda(\epsilon,\scaling)
\neq \lim_{\scaling\to\infty}\lim_{\varepsilon\to 0} \lambda(\epsilon,\scaling)= \lambda^\infty(a)
\enspace .
\]

%
%
%
%
%
%

%
%
%
%
%
%

\section{Monotonicity with respect to the division rate}

%
%
%
%
%
%

We first focus on the question of the monotonicity of the Floquet eigenvalue with respect to the division rate.
\subsection{A counter example in which the growth is not a monotone function of the division rate}
We next show that increasing the division rate may decrease the growth rate of the population. The counter example is based on the following exactly solvable model. We assume that the period is $T=1$, and fix $\alpha\in (0,1)$. It will be convenient to write
\[
I_1= [0,\alpha)+\mathbb{Z}, \qquad I_2 = [\alpha,1)+\mathbb{Z} \enspace,
\]
and
\[ \chi_j(t):=\bbbone_{I_j}(t)  ,\qquad j=1,2,\qquad \text{so that }
\chi_1+\chi_2 = 1 \enspace .
\]
We shall assume that the division rate
is of the form
\[ K(t,x)=\chi_1(t-x)K_1(t)+\chi_2(t-x)K_2(t) \enspace .
\]
Integrating along a characteristic, we get
\[
n(t,x)= n(t-x,0)\exp\big(-\int_0^xK(t-x+y,y)dy\big)
\]
and so
\[
n(t,x)=n(t-x,0)\exp\big(-\int_0^xK_j(t-x+y)dy\big), \qquad \text{if }t-x\in I_j \enspace .
\]
In other words, when an individual is born at time $t\in I_j$, then its next division will occur with a rate $K_j(t)$. This model has an intuitive interpretation. We may assume for instance
that $I_1$ represents the days, whereas
$I_2$ represents the nights. If being born during the night or during the day
influences the fertility (division rate) of the individuals, 
we arrive at this model.
We may then consider that there are actually two subpopulations $n_j(t,x)$, $j=1,2$, corresponding respectively to individuals born during the day 
or during the night.
We have 
\[n_j(t,x)=n(t,x)\chi_j(t-x),\qquad j=1,2
\]
They satisfy the equations:
$$
\bepa
\p_tn_j+\p_x n_j+K_j(t)n_j=0,\qquad j=1,2\\
n_j(t,0)=2\chi_j(t)\displaystyle\int_0^\infty (K_1(t)n_1(t,x)+K_2(t)n_2(t,x))dx \qquad j=1,2\enspace.
\eepa
$$
Denoting by 
\[ P_j(t)=\displaystyle\int_0^\infty n_j (t,x)dx
\]
the total population 
of type $j$ at time $t$, we get that the vector $P(t):=(P_1(t),P_2(t))^\top$
satisfies the ODE:
\[
\frac{d}{dt} P(t) = M(t)P(t)
\]
where 
\[
M(t):=\left(\begin{array}{cc}(2\chi_1(t)-1)K_1(t) & 2\chi_1(t)K_2(t)\\ 2\chi_2(t)K_2(t) & (2\chi_2(t)-1)K_2(t)\end{array}\right)
\]

We now assume that the functions $K_1(t)$ and $K_2(t)$ are constant 
during the day and during the night, meaning that
\[K_1(t)=a_{1}\chi_1(t)+b_1\chi_2(t),\qquad K_2(t)=a_2\chi_1(t)+b_2\chi_2(t),
\]
where the coefficients $a_i,b_i$ are constant. With this form of the coefficients, we obtain
\[M(t)=\chi_1(t)M_a +\chi_2(t)M_b
\]
where
\[M_a:=\left(\begin{array}{cc}a_1 & 2a_2\\ 0 & -a_2\end{array}\right),
\qquad 
M_b:= \left(\begin{array}{cc}-b_1 & 0\\ 2b_1 & b_2\end{array}\right)
\enspace .
\]
An important point here, still with the above interpretation is the following: $a_2$ and $b_1$ could be considered in some sense as transition coefficients between populations $n_1$ and $n_2$.

Since the linear dynamics $M(t)$ switches between the constant dynamics
$M_a$ and $M_b$, which are successively exercised during the intervals
$[0,\alpha)$ and $[\alpha,1)$, we get 
that 
\[
P(1)= \exp\big((1-\alpha)M_b\big)\exp\big(\alpha M_a\big)P(0) \enspace .
\]
Denoting by $\rho$ the Perron root (spectral radius) of a nonnegative
matrix, we arrive at the following expression for the Floquet eigenvalue
\begin{align}
\lambda_F = \log \rho\bigg( \exp\big((1-\alpha)M_b\big)\exp\big(\alpha M_a\big)\bigg) \enspace .
\end{align}
Let us now consider the situation in which $a_2=b_2=0$, meaning that
an individual born during the night is sterile (division will never occur).
Then, the second columns of the matrices $M_a$ and $M_b$ vanish,
and so 
\[ \exp\big((1-\alpha)M_b\big)\exp\big(\alpha M_a\big)
= \left(\begin{array}{cc}\exp(\alpha a_1 -(1-\alpha)b_1)& 0\\ \star & 1\end{array}\right)
\]
(the value of the off diagonal entry, denoted by $\star$, is irrelevant).
Since the spectral radius of a nonnegative triangular matrix is the maximum of its of its diagonal entries, we get
\[
\lambda_F =\max\big(\alpha a_1 -(1-\alpha)b_1, 0\big) \enspace .
\]
Hence, $\lambda_F$ is an increasing function of the division rate $a_1$,
but a {\em decreasing} function of the division rate $b_1$. This establishes
Proposition~\ref{theo-cex}. \qed

It should be noted that the preceding counter-example subsists
when $a_2$ and $b_2$ are sufficiently small, because the spectral
radius is a continuous function of the parameters. This is confirmed
by Figure~\ref{fig-cex}, in which the Floquet eigenvalue is plotted
as a function of $b_1$ and $b_2$ for $a_1=10$ and $a_2=0.1$.
\begin{figure}[htbp]\begin{center}
\includegraphics[scale=0.60]{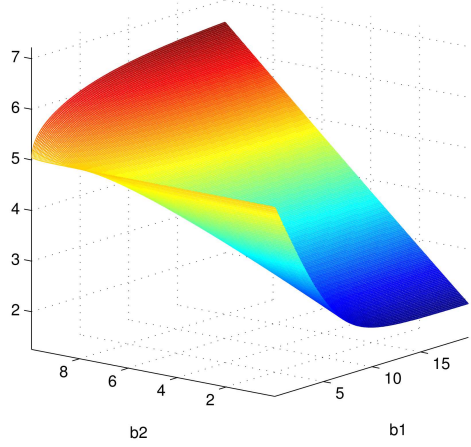}\end{center}
\caption{The counter example. The Floquet eigenvalue as a function of the division rates $b_1$ and $b_2$, when $a_1=10$ and $a_2=0.1$.
When $b_2$ is small, increasing $b_1$ (and so,
increasing the division rate $K(t,x)$) decreases the growth rate.}\label{fig-cex}
\end{figure}

\begin{remark}\label{rk-markov}
This counter example can be understood intuitively in Figure~\ref{fig-intuitive}. The values of the division rate $K(t,x)$ are represented (Left, the values are constant on each cell). Consider now 
a chain of descendants of the same individual (i.e., this
individual, one of his children, one of the children of this children, etc.).
This can be represented by a Markov process $X_t$ with state space in $\R_+$ (the age), which jumps from age $x$ to age $0$ with rate $K(t,x)$. Each time the path
hits the axis $x=0$ corresponds to a division. If the path hits $x=0$ during the night, since $b_2=0$, the corresponding individual is sterile (no division
will ever occur, which is reflected that the path is now an unending line
of slope $1$ in the $(t,x)$ plane). The production of sterile individuals
occurs when the process jumps from a cell labeled ``$b_1$'' to the $x$ axis,
and this occurs precisely with rate $b_1$, which explains intuitively
why the growth rate, $\lambda_F$, is a decreasing function of $b_1$.
\end{remark}
\begin{figure}[htbp]
	\centering
\begin{picture}(0,0)%
\includegraphics{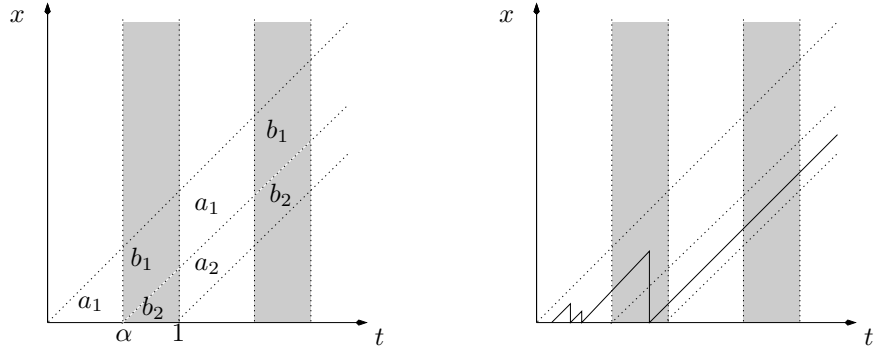}%
\end{picture}%
\setlength{\unitlength}{2072sp}%
\begingroup\makeatletter\ifx\SetFigFont\undefined%
\gdef\SetFigFont#1#2#3#4#5{%
  \reset@font\fontsize{#1}{#2pt}%
  \fontfamily{#3}\fontseries{#4}\fontshape{#5}%
  \selectfont}%
\fi\endgroup%
\begin{picture}(10245,4171)(2236,-6920)
\put(12466,-6856){\makebox(0,0)[lb]{\smash{{\SetFigFont{10}{12.0}{\rmdefault}{\mddefault}{\updefault}{\color[rgb]{0,0,0}$t$}%
}}}}
\put(8101,-2986){\makebox(0,0)[lb]{\smash{{\SetFigFont{10}{12.0}{\rmdefault}{\mddefault}{\updefault}{\color[rgb]{0,0,0}$x$}%
}}}}
\put(6616,-6856){\makebox(0,0)[lb]{\smash{{\SetFigFont{10}{12.0}{\rmdefault}{\mddefault}{\updefault}{\color[rgb]{0,0,0}$t$}%
}}}}
\put(2251,-2986){\makebox(0,0)[lb]{\smash{{\SetFigFont{10}{12.0}{\rmdefault}{\mddefault}{\updefault}{\color[rgb]{0,0,0}$x$}%
}}}}
\put(3061,-6406){\makebox(0,0)[lb]{\smash{{\SetFigFont{10}{12.0}{\rmdefault}{\mddefault}{\updefault}{\color[rgb]{0,0,0}$a_1$}%
}}}}
\put(3691,-5911){\makebox(0,0)[lb]{\smash{{\SetFigFont{10}{12.0}{\rmdefault}{\mddefault}{\updefault}{\color[rgb]{0,0,0}$b_1$}%
}}}}
\put(4456,-5236){\makebox(0,0)[lb]{\smash{{\SetFigFont{10}{12.0}{\rmdefault}{\mddefault}{\updefault}{\color[rgb]{0,0,0}$a_1$}%
}}}}
\put(5311,-4381){\makebox(0,0)[lb]{\smash{{\SetFigFont{10}{12.0}{\rmdefault}{\mddefault}{\updefault}{\color[rgb]{0,0,0}$b_1$}%
}}}}
\put(4456,-5956){\makebox(0,0)[lb]{\smash{{\SetFigFont{10}{12.0}{\rmdefault}{\mddefault}{\updefault}{\color[rgb]{0,0,0}$a_2$}%
}}}}
\put(4186,-6811){\makebox(0,0)[lb]{\smash{{\SetFigFont{10}{12.0}{\rmdefault}{\mddefault}{\updefault}{\color[rgb]{0,0,0}$1$}%
}}}}
\put(3511,-6811){\makebox(0,0)[lb]{\smash{{\SetFigFont{10}{12.0}{\rmdefault}{\mddefault}{\updefault}{\color[rgb]{0,0,0}$\alpha$}%
}}}}
\put(3826,-6496){\makebox(0,0)[lb]{\smash{{\SetFigFont{10}{12.0}{\rmdefault}{\mddefault}{\updefault}{\color[rgb]{0,0,0}$b_2$}%
}}}}
\put(5356,-5146){\makebox(0,0)[lb]{\smash{{\SetFigFont{10}{12.0}{\rmdefault}{\mddefault}{\updefault}{\color[rgb]{0,0,0}$b_2$}%
}}}}
\end{picture}%
	\caption{The counter example explained. The values of $K(t,x)$ (left). The night is shown in grey. A typical path (right, broken line). The child born during the night is sterile (final segment of the path), and the division coefficient $b_1$ determines the creation of this child.}
	\label{fig-intuitive}
\end{figure}
\begin{remark}
The system does not enter in the framework  considered in Appendix for the existence of the Floquet eigenvalue. Indeed, the preceding study merely determines
the growth rate of the ``aggregated'' population $P(t)$, which is determined
as the Perron eigenvalue of a matrix. However, the growth of $n(t,x)$ is readily
derived from the one of $P(t)$. To see this, it is convenient to introduce
the adjoint positive Floquet eigenvector of $M(t)$, whose existence
and uniqueness are guaranteed by Perron-Frobenius theory (as soon as the coefficients $a_1,a_2,b_1,b_2$ are positive):
$$
-\frac{d}{dt}\left(\begin{array}{cc}\phi_1 & \phi_2\end{array}\right)+\lambda \left(\begin{array}{cc}\phi_1 & \phi_2\end{array}\right)+\left(\begin{array}{cc}\phi_1 & \phi_2\end{array}\right)M(t)=0 .
$$
We have 
$$
\frac{d}{dt}(\phi(t),P(t))=\lambda (\phi(t),P(t))
$$
which can be read as, denoting $\phi(t,x)=\phi_1(t)\chi_1(t-x)+\phi_2(t)\chi_2(t-x)$, 
$$\frac{d}{dt}\displaystyle\int_0^\infty n(t,x)\phi(t,x)dx=\lambda \displaystyle\int_0^\infty n(t,x)\phi(t,x)dx,$$
It is straightforward to check that $\phi_1$ and $\phi_2$ are positive
(still assuming that the coefficients $a_1,a_2,b_1,b_2$ are positive).
Hence, there exists $M\geq m>0$ such that 
$m\leq \phi(t,x)\leq M$, and therefore, 
$$me^{\lambda t} \int n^0 dx\leq\int n(t,x)dx\leq Me^{\lambda t}\int n^0 dx.$$
This shows that the Floquet eigenvalue $\lambda$ computed from the aggregated
dynamical system does determine the behavior of the solutions. 
\end{remark}

\subsection{A sufficient condition for monotonicity}
The preceding counter example is due to a ``birth day penalty'' (the date
of birth of the individual influences critically its division rate). 
We shall avoid such a pathological behavior by making the following 
assumption:
\beq
\label{eq:K_condition}
(t,v)\mapsto \int_v^tK(s,s-v)ds\quad\text{is nonincreasing in }v.
\eeq
This condition is fulfilled when the division rate $K(t,x)\geq 0$ depends only on one of the two variables $t$ and $x$. It is also fulfilled if for any $t$, $K(t,\cdot)$ is a nondecreasing function (and particularly, in the case of separated variables $K(t,x)=\psi(t)B(x)$ when $B$ is nondecreasing). 

This is illustrated in Figure~\ref{fig-cond}. The condition requires
the integral of $K$ over the characteristic shown on the figure
to be nonincreasing in $v$. Note that this has a probabilist interpretation in terms of the Markov process $X_t$ introduced in the preceding section. Indeed,
\[\exp\big(-\int_v^tK(s,s-v)ds\big)
\]
represents the survival probability at time $t$ of an individual born at time $v$, i.e., the probability, conditional to $X_v=0$,
that the process makes no jump to point $0$ before time $t$. The condition
requires this probability to be a nondecreasing function
of $v$, in other words, that by delaying the birth, the survival
probability at a given time cannot decrease. 
\begin{figure}
\begin{center}
\begin{picture}(0,0)%
\includegraphics{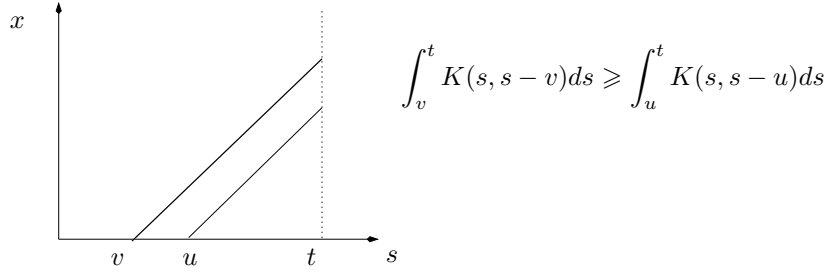}%
\end{picture}%
\setlength{\unitlength}{2072sp}%
\begingroup\makeatletter\ifx\SetFigFont\undefined%
\gdef\SetFigFont#1#2#3#4#5{%
  \reset@font\fontsize{#1}{#2pt}%
  \fontfamily{#3}\fontseries{#4}\fontshape{#5}%
  \selectfont}%
\fi\endgroup%
\begin{picture}(9421,3258)(2101,-6997)
\put(6616,-6856){\makebox(0,0)[lb]{\smash{{\SetFigFont{10}{12.0}{\rmdefault}{\mddefault}{\updefault}{\color[rgb]{0,0,0}$s$}%
}}}}
\put(5671,-6901){\makebox(0,0)[lb]{\smash{{\SetFigFont{10}{12.0}{\rmdefault}{\mddefault}{\updefault}{\color[rgb]{0,0,0}$t$}%
}}}}
\put(3331,-6901){\makebox(0,0)[lb]{\smash{{\SetFigFont{10}{12.0}{\rmdefault}{\mddefault}{\updefault}{\color[rgb]{0,0,0}$v$}%
}}}}
\put(4186,-6901){\makebox(0,0)[lb]{\smash{{\SetFigFont{10}{12.0}{\rmdefault}{\mddefault}{\updefault}{\color[rgb]{0,0,0}$u$}%
}}}}
\put(6796,-4741){\makebox(0,0)[lb]{\smash{{\SetFigFont{10}{12.0}{\rmdefault}{\mddefault}{\updefault}{\color[rgb]{0,0,0}$\displaystyle\int_v^tK(s,s-v)ds\geq \displaystyle\int_u^tK(s,s-u)ds $}%
}}}}
\put(2116,-4066){\makebox(0,0)[lb]{\smash{{\SetFigFont{10}{12.0}{\rmdefault}{\mddefault}{\updefault}{\color[rgb]{0,0,0}$x$}%
}}}}
\end{picture}%
\end{center}
\caption{The sufficient monotonicity condition: the survival probability at time $t$ (the exponential of the opposite of the integral) increases with the birth time $v$.}\label{fig-cond}
\end{figure}

\begin{theorem}\label{th:monotonicity}
Provided $K^1$ or $K^2$ satisfies \eqref{eq:K_condition} and $K^1\geq K^2$, then, for any $n^0\in L^1(\mathbb{R}_+)$, if we denote by $n^i$ the solution of \eqref{eq:periodic_division} with $K=K^i$, we have, for any $t\geq 0$,
$$\displaystyle\int_0^\infty n^1(t,x)dx\geq\displaystyle\int_0^\infty n^2(t,x)dx.$$
\end{theorem}
To establish this result, we introduce the following system of PDEs:
\beq\label{eq:division_generation}
\bepa
\p_tn_i+\p_xn_i+K(t,x)n_i=0,\qquad i\geq 0,\\[0.3cm]
n_0(t,0)=0 \\
n_{i}(t,0)=\dis\displaystyle\int_0^\infty K(t,x)n_{i-1}(t,x)dx,\qquad i\geq 0\\[0.3cm]
n_0(t=0,x)=n^0(x),\qquad n_i(t=0,x)=0,\quad\text{for }i\geq 1. 
\eepa
\eeq

We easily check that $n=\sum_{i\geq 0} 2^i n_i$ is the solution to \eqref{eq:periodic_division}. Intuitively, the term $2^i n_i$ represents the number of individuals of the ``generation'' $i$, and so, we shall refer to this model as the 
``generational model'' in the sequel.
We denote by $(n_i^1)_{i\geq 0}$, (resp.\ $(n_i^2)_{i\geq 0}$) the solution to \eqref{eq:division_generation} with $K=K^1$ (resp.\ $K=K^2$). We also introduce:
\[
S_j(t):=\sum_{i\geq j}\displaystyle\int_0^\infty n_i(t,x)dx,\qquad \text{for } j\geq 0
\enspace .
\]
A short computation leads to 
\[ \displaystyle\int_0^\infty n(t,x)dx=S_0(t)+\sum_{j \geq 1} 2^{j-1}S_j(t)
\enspace .
\]
\begin{lemma}\label{lem:Sj}
We have:
\begin{align}
S_0(t)&=S_0(0)=\displaystyle\int_0^\infty n^0(x) dx \enspace,
\\
S_1(t)&=\displaystyle\int_0^\infty n^0(x)\bigg(1-\exp\big(-\int_0^t K(s,x+s)ds\big) \bigg)dx \enspace,
\end{align}
and for $j\geq 1$,
\begin{align}
S_j(t)&=\int_{s=0}^t n_{j-1}(s,0)\bigg(1-\exp\big(-\int_0^{t-s}K(s+y,y)dy\big)\bigg)ds\enspace .\label{e-final}
\end{align}
\end{lemma}
\begin{proof}
It is straightforward to show that for every $j\geq 0$,
\[\f{d}{dt}S_j(t)=n_j(t,0) \enspace,
\]
and so
\begin{align} S_j(t)=\int_0^t n_j(s,0)ds+S_j(0) \enspace .
\label{e-bilandemasse}
\end{align}
This leads immediately to  the first statement, since for all $s>0$,
$n_0(s,0)=0$. 

To establish the next statements, we shall use the characteristics:
\begin{align}
n_0(t,x)&=0,\quad \text{if }x\leq t,\qquad \nonumber\\
n_0(t,x)&=n^0(x-t)\exp\big(-\int_0^tK(s,x-t+s)ds\big),\quad\text{if }x\geq t  \enspace,\label{e-carac1}
\end{align}
and for $i\geq 1$,
\begin{align}
n_i(t,x)&=0, \quad \text{if }t\leq x, \qquad \nonumber\\ 
n_i(t,x)&=n_i(t-x,0)\exp\big(-\int_0^xK(t-x+y,y)dy\big),\quad\text{if }x\leq t \enspace. \label{e-carac2}
\end{align}
The lack of symmetry between the cases $i\geq 1$ and $i=0$ is due to the boundary conditions: $n_0(0,x)=n^0(x)$ is known, whereas for $i\geq 0$,
$n_i(t,0)$ can be inductively assumed to be known, using $n_{i-1}$.

Since $n_1(0,x)=0$, we have $S_1(0)=0$. Using~\eqref{e-bilandemasse}, we get
\begin{align*}
S_1(t)&=\int_{s=0}^t n_1(s,0)ds= \int_{s=0}^t \int_{x=0}^\infty K(s,x)n_{0}(s,x) dx ds \\
&=\int_{s=0}^t\int_{x=s}^\infty K(s,x)n^0(x-s)\exp\big(-\int_0^sK(u,x-s+u)du\big)dxds\\
&=\int_{s=0}^t\int_{y=0}^\infty K(s,y+s)n^0(y)\exp\big(-\int_0^sK(u,y+u)du\big)dyds \enspace,
\end{align*}
and, interchanging the order of integration, 
\begin{align*}
S_1(t)&=\int_{y=0}^\infty n^0(y)\int_{s=0}^tK(s,y+s)\exp\big(-\int_0^sK(u,y+u)du\big)dsdy\\
&=\displaystyle\int_0^\infty n^0(y)\bigg(1-\exp\big(-\int_0^t K(s,y+s)ds\big)\bigg)dy,\end{align*}
which is the second statement. 

Assume now that $j\geq 2$. Arguing as above, but using this time~\eqref{e-carac2} instead of~\eqref{e-carac1}, we get 
\begin{align*}
S_j(t)&=\int_{s=0}^t \int_{x=0}^\infty K(s,x)n_{j-1}(s,x) dx ds \\
&= \int_{s=0}^t \int_{x=0}^s K(s,x)n_{j-1}(s-x,0)\exp\big(-\int_0^xK(s-x+y,y)dy\big) dx ds \\
&= \int_{x=0}^t\int_{s=x}^tK(s,x)n_{j-1}(s-x,0)\exp\big(-\int_0^xK(s-x+y,y)dy\big)dsdx\\
&=\int_{x=0}^t\int_{u=0}^{t-x}K(u+x,x)n_{j-1}(u,0)\exp\big(-\int_0^xK(u+y,y)dy\big)dudx\\
&=\int_{u=0}^tn_{j-1}(u,0)\int_{x=0}^{t-u}K(u+x,x)\exp\big(-\int_0^xK(u+y,y)dy\big)dxdu,
\end{align*}
which leads to the third statement.
\end{proof}

Theorem~\ref{th:monotonicity} will be obtained as an immediate consequence of the following lemma.
\begin{lemma}\label{lem-majorization}
Suppose $K^1$ satisfies \eqref{eq:K_condition}, then 
\begin{align*}
K^2\geq K^1&\implies S_j^2(t)\geq S_j^1(t), \qquad \forall t,j\\
K^2\leq K^1&\implies S_j^2(t)\leq S_j^1(t), \qquad \forall t,j\enspace .
\end{align*}
\end{lemma}
\begin{proof} The proof is performed by induction on $j$. We shall only
consider the case in which $K^2\geq K^1$, the other case being similar.

For $j=0,1$, the conclusion follows readily from the two first statements
in Lemma~\ref{lem:Sj}. Assume now that $j\geq 2$. 
We set
\[ Q^k(t,v):=1-\exp\big(-\int_0^{t-v}K^k(v+y,y)dy \big)
\]
Using~\eqref{e-final}, we get
\begin{align*}
 S_j^1(t)-S^2_j(t)&=\int_0^tn_{j-1}^1(v,0)Q^1(t,v)dv-\int_0^tn_{j-1}^2(v,0)Q^2(t,v)dv,\\
 &\!\!\!\!\!\!\!\!\!\!\!\!\!\!\!\!\!\!\!\!\!\!\!\!\!\!\!\!
=\int_0^t(n_{j-1}^1(v,0)-n_{j-1}^2(v,0))Q^1(t,v)dv+\int_0^tn_{j-1}^2(v,0)(Q^1(t,v)-Q^2(t,v))dv\\
 &\!\!\!\!\!\!\!\!\!\!\!\!\!\!\!\!\!\!\!\!\!\!\!\!\!\!\!\!=\big[(S_{j-1}^1(v)-S_{j-1}^2(v))Q^1(t,v)\big]^t_0-\int_0^t(S_{j-1}^1(v)-S_{j-1}^2(v))\f{d}{dv}Q^1(t,v)\\
 &\qquad +\int_0^tn_{j-1}^2(v,0)(Q^1(t,v)-Q^2(t,v))dv,\\
 &\!\!\!\!\!\!\!\!\!\!\!\!\!\!\!\!\!\!\!\!\!\!\!\!\!\!\!\!=-\int_0^t(S_{j-1}^1(v)-S_{j-1}^2(v))\f{d}{dv}Q^1(t,v)+\int_0^tn_{j-1}^2(v,0)(Q^1(t,v)-Q^2(t,v))dv \enspace .
\end{align*}
 The condition \eqref{eq:K_condition} ensures the nonpositivity of $\f{d}{dv}Q^1(t,v)$. Moreover, $K^2\geq K^1$ readily implies $Q^2\geq Q^1$. Using the induction
hypothesis, we deduce that $S_j^2\geq S_j^1$.
\end{proof}
\begin{remark}
Lemma~\ref{lem-majorization} has an interpretation
in terms of stochastic order or majorization~\cite{marshall79}. 
Recall that for $\mathbb{R}$-valued random variables $Z^1,Z^2$,
the stochastic order $\geq_{st}$ is such that $Z^2\geq_{st} Z^1$ holds
if and only if $\mathbb{P}(Z^2\geq t) \geq \mathbb{P}(Z^1 \geq t)$ for all $t\in \mathbb{R}$.
Now, let $X(t)$ denote the Markov process representing
the age of a single individual, defined in Remark~\ref{rk-markov},
and let $Y(t)$ denotes its number of jumps at time $t$, i.e.,
the ``generation'' of an individual.
Then, $S_j(t)= \sum_{i\geq j} \int_0^\infty n_i(x,t)dt = \mathbb{P}(Y(t) \geq j)$,
and so, denoting by $Y^1(t),Y^2(t)$ the two processes obtained
in this way from $K^1,K^2$, we see that the properties
of Lemma~\ref{lem-majorization} are equivalent to 
\begin{align*}
K^2\geq K^1&\implies Y^2(t)\geq_{st} Y^1(t), \qquad \forall t\\
K^2\leq K^1&\implies Y^2(t)\leq_{st} Y^1(t), \qquad \forall t\enspace .
\end{align*}
\end{remark}
 
\section{Asymptotics of the growth rate}\label{sec:asymp}
We now assume that the division
rate is of the form
\[
K(t,x)=\scaling \psi(t)B(x)\bbbone_{[a,+\infty[}(x) 
\]
an determine the limit of the Floquet eigenvalue of System~\eqref{eq:periodic_division} as the division rate tends to infinity, i.e., as the scaling
parameter $\scaling$ tends to infinity.
\subsection{\texorpdfstring{When both $\psi$ and $B$ are positive: proof of Theorem~\ref{Th:log2sura}}{When both psi and B are positive: proof of Theorem~\ref{Th:log2sura}}}
The Floquet eigenvalue $\lambda_F^\scaling$ satisfies

\begin{equation}\label{eq:Floquet_B}
\left\lbrace\begin{array}{l}
\frac{\partial}{\partial t}N^\scaling(t,x)+\frac{\partial}{\partial x}N^{\scaling}(t,x) +\big[\lambda_F^{\scaling}+ \scaling\psi(t)B(x)\bbbone_{[a,+\infty[}(x) \big]N^{\scaling}(t,x)=0, \\[0.3cm]
N^{\scaling}(t,0)=2\scaling\psi(t) \int_{a}^\infty B(x)N^{\scaling}(t,x)dx,\\[0.3cm]
N^{\scaling}>0, \; \text{$T$-periodic}.
\end{array}\right.
\end{equation}
The existence of the Floquet eigenvector $N^\scaling$ is derived in the appendix
as a consequence of the Krein-Rutman theorem. In particular,
the function $B$ is required to be positive and of infinite integral.
We normalize the Floquet eigenvector by requiring that 
\begin{equation}
\int_0^T\displaystyle\int_0^\infty N^{\scaling}(t,x)dxdt=1 \enspace .
\end{equation}

To establish Theorem~\ref{Th:log2sura}, we shall think of the Floquet
eigenvalue $\lambda_F^\scaling=\lambda_F^\scaling(\psi)$ as a function of
the time modulation $\psi$. 
Since the division rate is a nondecreasing function of $\psi$, using the comparison principle (Theorem~\ref{th-comparison}), 
we get
\[
\lambda_F^\scaling(\underline{\psi}) \leq \lambda_F^\scaling(\psi) \leq \lambda_F^\scaling(\overline{\psi})\enspace ,
\]
where $\underline{\psi}$ and $\overline{\psi}$ denote the constant functions equal to the minimum or maximum of $\psi$, respectively.
When the function $\psi(t)$ is equal to a constant $\alpha$, the Floquet eigenvector can be chosen to be independent of $t$, so that $N^\scaling(t,x)=N^\scaling(x)$. Then, the Floquet
eigenvalue $\lambda^\scaling_F(\alpha)$ is actually an ordinary eigenvalue.
Indeed, it is the only scalar $\mu$ solution of the system
\begin{equation}\label{eq:Floquet_Bb}
\left\lbrace\begin{array}{l}
\frac{d}{dx}N^{\scaling}(x) +\big[\mu+ \scaling\alpha B(x)\bbbone_{[a,+\infty[}(x) \big]N^{\scaling}(x)=0, \\[0.3cm]
N^{\scaling}(0)=2\scaling\alpha \int_{a}^\infty B(x)N^{\scaling}(x)dx,\\[0.3cm]
N^{\scaling}>0
\end{array}\right.
\end{equation}
Integrating the latter differential equation, we get
\[
N^\scaling(x)=N^\scaling(0)\exp\big(-\mu x -\scaling \alpha \int_a^x B(y) dy\big) \enspace  .
\]
Using the boundary condition, and eliminating $N^\scaling(0)$, we determine
$\mu$ via the following implicit equation:
\[
2\int_a^\infty \scaling\alpha B(x)\exp\big(-\mu x -\scaling\alpha \int_a^x B(y)dy\big)dx=1
\]
Since $\mu=\lambda_F^\scaling(\alpha)\geq 0$ (in fact, $\mu>0$, see Theorem~\ref{Th:Existence_Floquet} in the appendix), we have
\[ 
1\leq 2\exp(-\mu a)\int_a^\infty \scaling\alpha B(x)\exp\big( -\scaling\alpha\int_a^x B(y)dy\big)dx=2\exp(-\mu a)
\]
and, taking $\alpha:=\overline{\psi}$, we get
\[
\lambda_F^\scaling(\psi)\leq \lambda_F^\scaling(\overline{\psi})\leq \frac{\log 2}{a}
\enspace .
\]
To bound the Floquet eigenvalue from below, we perform
an integration by parts, so that:
\[
 1=2\exp(-\mu a)-\mu\int_a^\infty \exp\big(-\mu  x -\scaling\alpha\int_a^x B(y)dy\big)dx \enspace. 
\]
Take now $\alpha:=\underline{\psi}$. 
Using the fact that $0<\lambda_F^\scaling(\underline{\psi})$
and a dominated convergence argument, we deduce that
\[ \int_a^\infty \exp\Big(-\lambda_F^\scaling(\underline{\psi}) x -\scaling\underline{\psi}\int_a^x B(y)dy\Big)dx
\leq \int_a^\infty \exp\Big(-\scaling\underline{\psi}\int_a^x B(y)dy\Big)dx\xrightarrow[\scaling\rightarrow +\infty]{} 0
\]
and so
\[2\exp(-\lambda_F^{\scaling}(\underline{\psi}) a)\xrightarrow[\scaling\rightarrow +\infty]{} 1,
\]
which shows that $\liminf_{\scaling\to\infty}\lambda_F^\scaling(\psi)\geq \lim_{\scaling\to\infty}\lambda_F^\scaling(\underline{\psi})=\log 2/a$, concluding
the proof of the theorem. \qed

We give a numerical illustration of this theorem in Figure~\ref{fig:limit152050}. These numerical simulations were produced along the lines of~\cite{Lepoutre_mmnp} using a monotone finite difference
scheme. The Floquet eigenvalue was computed by applying the power algorithm
to the one day discretized evolution operator.
\begin{figure}[htpb]
	\centering

	\includegraphics[width=0.50\textwidth]{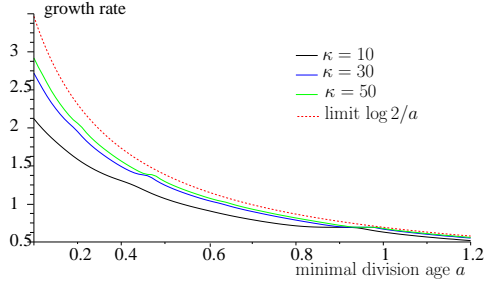}
\caption{Convergence of the Floquet eigenvalue to $\log 2/a$.}	\label{fig:limit152050}
\end{figure}

%
%
%

\subsection{\texorpdfstring{When $\psi$  can vanish: proof of Theorem~\ref{Th:devil}}{When psi  can vanish: proof of Theorem 1.2}}

When $\psi$ can vanish, the conclusion of Theorem~\ref{Th:log2sura} fails; its proof uses the assumption that $\min(\psi)>0$. 
Then, we next show that more complex behaviors appear. In particular, the limit of $\lambda_F^k$ can be a discontinuous,
staircase like, function of the majority age $a$, and we
shall see that the limits $\scaling\rightarrow +\infty$ and $\psi\rightarrow 0$ (locally) may not commute.
\subsubsection{Intuitive derivation of the formula via a deterministic jump Markov process}\label{sec:derivation}
%
%
We consider the original PDE (\ref{eq:periodic_division}),  taking $T=1$,
\begin{equation}\label{eq:Malthus_d=0}
\left\lbrace\begin{array}{l}
\frac{\partial}{\partial t}n(t,x) + \frac{\partial}{\partial x}n(t,x) 
+\scaling\psi(t)\bbbone_{[a,\infty[}(x)n(t,x) = 0, \\[0.3cm] 
n(t,0) = 2\scaling\psi(t)\int_{a}^{\infty}n(t,x)dx.
\end{array}\right.
\end{equation}
where $\psi$ is a 1-periodic square wave,  such that if $\floor{t}$ is the integer part of $t$,
\begin{equation}\label{as:psi_square_wave}
\exists \, 0<\tau <1,\quad \psi= t\mapsto \bbbone_{[0,\tau[}(t-\floor{t}).
\end{equation} 

We next derive the limit of the growth
rate $\lambda$ when $\scaling\rightarrow+\infty$ (within this section, $\lambda$ denotes the Floquet eigenvalue $\lambda_F^\scaling$, thought of as a function
of $a$ and $\psi$). This derivation will be purely formal for the moment, 
leading to an ansatz the validity of which will be proved in the next sections.

Intuitively, whenever the cell is old enough in the cycle ($x>a)$ and
the time is favorable ($0\leq t-\floor{t}\leq \tau$), a division may occur.
The idea is to consider that at time $t=0$, we have only cells 
of age $0$ ($n(0,\cdot)=\delta_0$), and to look formally
for an asymptotic solution 
\[
n(t,\cdot) \sim  2^{m(t)} \delta_{\alpha(t)} , \qquad \kappa \to \infty
\]
of the PDE~\eqref{eq:Malthus_d=0}, meaning that the population 
at time $t$ consists mostly of $2^{m(t)}$ individuals all of age $\alpha(t)$,
as $\kappa\to\infty$. 
If for some integer time $K$, we have $\alpha(K)=0$, 
so that $n(K,\cdot) \sim 2^{m(K)}\delta_0$, using the time-periodicity
and linearity of~\eqref{eq:Malthus_d=0}
we may expect 
that $n(p K,\cdot) \sim  2^{p m(K)}\delta_0$, for all integers
$p\geq 1$. Commuting the limits as $\kappa\to\infty$ and $p\to\infty$, we arrive at the following expression
for the limit of the growth rate as $\kappa \to \infty$,
\begin{align}
\label{e-def-lambdainfty}
\lambda^\infty(a):=
 \frac{m(K)}{K \log 2}
= \lim_{t \to \infty} \frac{m(t)}{t \log 2} \enspace .
\end{align}
The idea is now that the trajectory $(\alpha(t))_{t\geq 0}$, representing
the age of the population, is produced
by a degenerate (deterministic) time-dependent \cadlag\ jump process, in which 
the state space is $\mathbb{R}_+$, and a particle moves to the right
with unit speed, untill it reaches an age 
and a time at which division is permitted,
meaning that $\alpha(t)\geq a$ and $t \in [0,\tau[+ \mathbb{N}$.
At every time satisfying this condition, it jumps to point $0$
(age is reset to zero and a division occurs).
 The number $m(t)$ counts the number of
jumps up to time $t$ and represents the number of divisions
of the cell which occured up to that time.

It suffices for our argument to consider the case in which
$\alpha(0)=0$. Then, the division times can  be readily computed:
if we start from age $0$ at time $0$, 
then, a division should occur at every instant
$a,2a,\dots, ka$ as long as it is permitted,
meaning that $a,2a,\dots,ka\in [0,\tau[ + \mathbb{N}$. 
Let
\begin{align}
K_a:=  \inf \{ k\in \mathbb{N} ;\; ka \in [\tau,1[ + \mathbb{N} \} \enspace,
\label{e-def-ka}
\end{align}
so that $K_aa$ is the first time at which a division is not permitted,
or $K_a = \infty$.
When, for the first time, we reach a possible division time $K_aa$ that is not permitted, then the corresponding division occurs later, namely at the next period : $\lceil K_aa\rceil$, where $\lceil \cdot\rceil$ denotes the upper integer part. Then, at time $\lceil K_aa\rceil $ the population divides (for the $K_a$th time) and we have $n(\lceil K_aa\rceil ,\cdot)=2^{K_a}\delta_0$. Therefore, \eqref{e-def-lambdainfty}
yields
\begin{align}
e^{\lambda^\infty(a)\lceil K_aa\rceil }=2^{K_a},\qquad\lambda^\infty(a)=\frac{K_a}{\lceil K_aa\rceil}\log 2 \enspace .
\label{e-def2-lambdainfty}
\end{align}
When $K_a=\infty$, we shall adopt the convention that $\lambda^\infty(a)=(\log 2 )/a$. 

This is illustrated in Figure~\ref{fig:devil}.

\begin{figure}[!h]\label{fig:devil}
\begin{center}\includegraphics[scale=0.4]{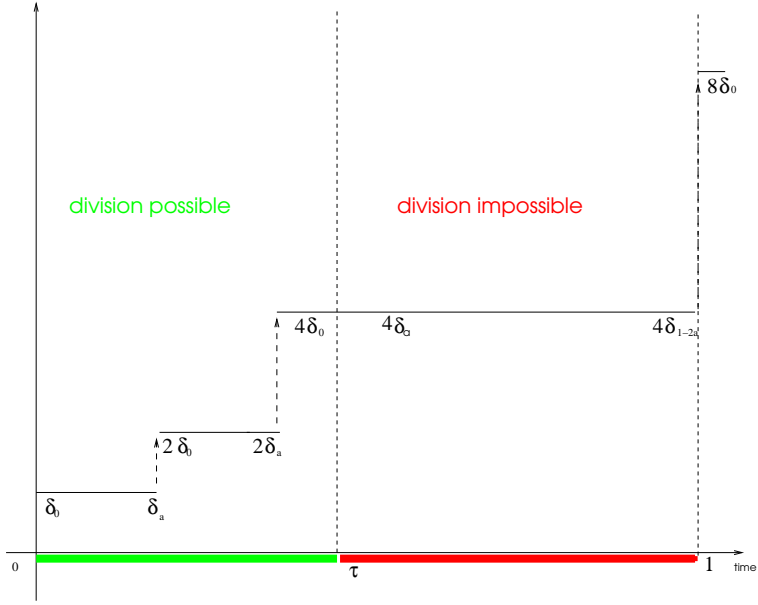}
\end{center}
\caption{Example of the derivation of the formula. Here $2a<\tau<3a<1$ and we obtain $n(1,\cdot)=8n(0,\cdot)$. Hence we expect
$\lambda^{\infty}(a)=3\log(2)$}
\end{figure}

%
%
%

%
%

%
%
%

Working with a semi-open interval $[\tau,1[$ in~\eqref{e-def-ka} is 
consistent with the \cadlag\ nature of the process $\alpha(t)$, however,
it leads to the possibility that $K_a=\infty$, which
is somehow unpleasant. Hence, we will use
in the sequel an equivalent formulation of $\lambda^\infty$, obtained
by replacing $[\tau,1[$ by the closed interval $[\tau,1]$,
\[
N_a:=  \min \{ k\in \mathbb{N} ;\; ka \in W_\tau  \} \enspace,
\qquad 
W_\tau:= [\tau,1] + \mathbb{N} \enspace.
\]
\begin{lemma}\label{lemma-reformulate}
We have $N_a<\infty$. Moreover, 
\begin{align}
\lambda^\infty(a) = \frac{N_a}{\lceil N_a a\rceil} \log 2 \enspace.
\label{e-nicelambdainfty}
\end{align}
\end{lemma}
\begin{proof}
If $a$ is irrational,
the finiteness of $N_a$
follows from the density of the sequence $ka$ modulo $1$, in $[0,1[$; 
whereas if $a$ is rational, we have $ka \in 1 +\mathbb{N}$ for
some integer $k$, so that $N_a\leq k$ is finite. 
It remains to show that the expressions~\eqref{e-def2-lambdainfty} and~\eqref{e-nicelambdainfty}
coincide.
Assume first that $N_a a\in [\tau,1[+\mathbb{N}$.
Then, by definition of $N_a$, $ka\in ]0,\tau[+\mathbb{N}$ for $1\leq k<N_a$,
and so, $ka\in [0,\tau[$ for $0\leq k<N_a$, which implies
that $K_a=N_a$. Assume now that $N_a a \in 1+\mathbb{N}$.
Then, we deduce from $ka\in ]0,\tau[+ \mathbb{N}$,
for $1\leq k<N_a$ that $(N_a+k) a \in ]0,\tau[+\mathbb{N}$,
for all $1\leq k<N_a$, and so $K_a=\infty$, which by our convention
regarding~\eqref{e-def2-lambdainfty} leads to~\eqref{e-nicelambdainfty}.
\end{proof}

In the remaining subsections, we show that $\lim_{\kappa\to\infty} \lambda^\kappa(a)$
does coincide with the function $\lambda^\infty(a)$ (Theorem~\ref{Th:devil}).
Throughout the proof, the parameter $\tau$ will be fixed.
The proof of the theorem relies on the following strategy:
\begin{enumerate}
\renewcommand{\theenumi}{\roman{enumi}}
\item prove that the ansatz $\lambda^\infty(a)$ (staircase formula) is an upper bound for $\lambda^{\scaling}(a)$, for every $\kappa>0$, (Section~\ref{sec:upper_bound});
\item derive technical properties of the limit and of $\lambda^\infty(a)$;
in particular monotonicity and right continuity, (Section~\ref{sec:technical});
\item assume that $a$ is not the right end of a step of $\lambda^\infty$, 
\item consider the generational formulation of the division equation with initial data $N^{\scaling}(0,x)$ and show that at time $\lceil N_a a\rceil$,
almost every individual has reached generation $N_a$ and deduce the convergence for such an $a$.
\item conclude that the limit of $\lambda^\kappa$ as $\kappa\to\infty$ is given by the staircase formula of $\lambda^\infty(a)$, for all $a>0$, using the fact that both $\lambda^\infty$ and $\lim_{\kappa \to \infty} \lambda^\kappa$ are right continuous functions of the parameter $a$.
\end{enumerate}

%
%
%

\subsubsection{\texorpdfstring{Preliminaries: technical properties of $\lambda^\infty$ and of $\lim_{\kappa\to\infty} \lambda^\kappa$}{Preliminaries}}\label{sec:technical}

We shall need a number of technical observations regarding the function
$\lambda^\infty$ defined in~\eqref{e-def2-lambdainfty}.

\begin{lemma}
The function $a\mapsto \lambda^\infty(a)$ is nonincreasing.
\end{lemma}
\proof
Let $(\alpha(t))_{t\geq 0}$ denote the trajectory of the 
jump process defined above, starting from $\alpha(0)=0$,
and let $m(t)$ denote the number of jumps (divisions) up to time $t$ (included).
We have
\begin{align*}
m(t)&=\sup\left\{k\in\mathbb{N};\;\exists 0\leq a_1\leq \dots\leq a_k\leq t,\; a_{i+1}-a_i\geq a,\; a_i\in [0,\tau[+\mathbb{N}\right\} \\
& \qquad\qquad \qquad\qquad\qquad\qquad\qquad\quad\;\text{(number of divisions since time $0$)}
\end{align*}
Since the constraint $a_{i+1}-a_i\geq a$ gets stronger as $a$ increases, it is then obvious that $m$ is a nonincreasing function of $a$. We have,
by definition of $\lambda^\infty$,
\beq 
\lim_{t\rightarrow+\infty}\frac{m(t)}{t}= \frac{N_a}{\lceil N_aa\rceil}=\frac{\lambda^\infty}{\log 2}.
\eeq
Since for every $t$, $m(t)$ is a nonincreasing  function of $a$, the same is true for $\lim_t m(t)/t = \lambda^\infty/\log 2$.\qed

\begin{lemma}
The limit $\ell(a)=\lim_{\scaling\rightarrow \infty} \lambda^{\scaling}(a)$ satisfies 
$$\ell(a)=\sup_{\scaling>0}\lambda^{\scaling}(a) \enspace,$$
and the map $a\mapsto \ell(a), $ is nonincreasing and right continuous.
\end{lemma}
\begin{proof}
Since $\lambda^{\scaling}(a)\leq \log 2/a$, and 
since, by the comparison principle (Theorem~\ref{th-comparison}),
$\lambda^{\scaling}(a)$ is a nondecreasing function of $\scaling$,
$$\ell(a)=\sup_{\scaling>0}\lambda^{\scaling}(a) \leq \frac{\log 2}a \enspace .$$
Furthermore, since for any fixed $\scaling$, $\lambda^{\scaling}(a)$ is nonincreasing with $a$, so is $\ell(a)$. Finally, since for any fixed $\scaling>0$, $\lambda^{\scaling}(a)$ is continuous with respect to $a$, $\ell(a)$, which
is a supremum of a family of lower semi-continuous functions, is lower semi-continuous. Combined with monotonicity, this leads to the convenient property:
$$\forall a>0,\qquad \ell(a)=\liminf_{x\rightarrow a} \ell(x)=\lim_{x\rightarrow a+0}\ell(a).$$
That is, $\ell(a)=\lim_{\scaling\rightarrow \infty} \lambda^{\scaling}(a)$ is right continuous.\end{proof}

\begin{lemma}
The function $a\mapsto\lambda^\infty(a)$ 
is  right continuous. 
\end{lemma}
\begin{proof}
By Lemma~\ref{lemma-reformulate}, 
$$
N_a a=\min W_\tau\cap\mathbb{N}a=:f(a),\quad \lambda^\infty(a)=\frac{N_a}{\lceil N_a a\rceil}\log 2=\frac{f(a)}{a\lceil f(a)\rceil}\log 2,
$$
where we just denote $f(a)=N_aa$. 
To show that the function $\lambda^\infty$ is right continuous, it suffices to 
prove that
$$g(a)=\frac{f(a)}{\lceil f(a)\rceil},$$
is right continuous. If $g(a)\not=1$, this means that $f(a)\not=\lceil f(a)\rceil$, and therefore, for $\varepsilon>0$ small enough, we can have $n(a+\varepsilon)\in [\lfloor na\rfloor, \lfloor na\rfloor+\tau[$ for all $n<N_a$ and  $N_a=N_{a+\varepsilon}$, $f(a+\varepsilon)=f(a)+N_a\varepsilon$ and $\lceil f(a)\rceil=\lceil f(a+\varepsilon)\rceil$. 
Thereby, for $\varepsilon>0$ small enough, we have $g(a+\varepsilon)=g(a)+\frac{N_a}{\lceil f(a)\rceil}\varepsilon$, which leads to the right continuity of $g$ at $a$. If $g(a)=1$, then, we can check that 
$$
\lim_{\varepsilon\to 0^+}f(a+\varepsilon)=+\infty.
$$
and so
$$
\lim_{\varepsilon\to 0^+}g(a+\varepsilon)=1.
$$
\end{proof}

These observations leads to the following useful property (stating basically that $\lambda^\infty$ is piecewise constant)
\begin{lemma}\label{lem:interval}
For all $a>0$, $I_a=\{\lambda^\infty(x)=\lambda^\infty(a)\}$ is an interval with nonempty interior. Moreover, $\inf I_a\in I_a$. Finally, if we denote 
$$E_\tau:\{a>0,\;a=\sup I_a\}\enspace ,$$ we have 
\begin{itemize}
\item if $a\in E_\tau$ then $\lambda^\infty$is continuous at $a$,
\item if $a\in E_\tau$ then  $N_aa$ is an integer.
\end{itemize}
\end{lemma}
We shall see that the set $E_\tau$ contains exactly the right ends of steps of $\lambda^\infty$ where no jump occur. An infinite number of steps accumulate at the right of every point of $E_\tau$. On Figure~\ref{fig:figure_devil_2tiers}, we can see for instance that $1/2$ belongs to $E_{2/3}$
\begin{proof}
{\em Step 1: $I_a$ is an interval}

Firstly, this set is necessary an interval by monotonicity : if $a'\in I_a$, then for any convex combination of $a,a'$, we have $$\lambda^\infty(a)=\lambda^\infty(\max(a,a'))\leq \lambda^\infty(\theta a+(1-\theta)a')\leq\lambda^\infty(\max(a,a'))=\lambda^\infty(a). $$
To prove that it is nontrivial, we use the definition of $N_a$. By definition $ a, 2 a ,\dots (N_a-1)a\not\in W_\tau$. As $W_\tau$ is a closed set, its complementary is open and thus for $\varepsilon>0$ small enough, 
$$
]a-\varepsilon,a+\varepsilon[\cap W_\tau=]2(a-\varepsilon),2(a+\varepsilon)[\cap W_\tau=](N_a-1)(a-\varepsilon),(N_a-1)(a-\varepsilon)[\cap W_\tau=\emptyset.
$$
Finally, for $\varepsilon>0$ small enough, we have necessarily $]N_a(a-\varepsilon),N_a a]\subset W_\tau$ or $[N_a a,N_a(a+\varepsilon)[\subset W_\tau$ (possibly both). Thereby, for any $a$, there exists $\varepsilon>0$ such that $[a,a+\varepsilon[$ or $]a-\varepsilon,a]$ is contained in $I_a$ which has then a nonempty interior (contains an interval of size $\varepsilon$). The fact that $\inf I_a$ belongs to $I_a$ is just a consequence of right continuity.

\medskip\noindent{\em Step 2: continuity on $E_\tau$}

The continuity follows from right continuity of $\lambda^\infty$ and the fact that in this case,  $\lambda^\infty(\sup I_a)=\lim_{x\rightarrow \sup I_a -0} \lambda^\infty(x)$, that is $\lambda^\infty$ is also left continuous. 

\medskip\noindent{\em Step 3: $N_aa=\lceil N_a a \rceil$, forall $a\in E_\tau$}

Assume that $a$ satisfies  $\lceil N_a a \rceil-1+\tau \leq N_aa<\lceil N_a a \rceil$. Then, for some $\varepsilon>0$, we have $\lceil N_a a \rceil-1+\tau\leq N_aa'<\lceil N_a a \rceil$ for any $a'\in ]a,a+\varepsilon[$. Thereby, we can claim by definition of $N_a$  that
$$N_{a'}\leq N_a, \qquad\lceil N_{a'} a' \rceil\leq \lceil N_{a}a\rceil.$$ Suppose now that in addition that $a\in E_\tau$. By definition, as $a=\sup I_a$, we have $\lambda^\infty(a')\not=\lambda^\infty(a)$ for $a'\in]a,a+\varepsilon[$. Furthermore for $a'\in]a,a+\varepsilon[$, $\lambda^\infty$ takes its values in 
$$
\left\{\frac{N}{p}\log 2,\quad N\leq N_a,\quad p\leq \lceil N_a a \rceil\right\},
$$
therefore, for $a'\in]a,a+\varepsilon[$, we have $$|\lambda^\infty(a)-\lambda^\infty(a')|\geq \frac{1}{(\lceil N_a a \rceil)^2}\log 2$$ and this contradicts the continuity of $\lambda^\infty$ at $a$.\end{proof}

The set $E_\tau$ has a last interesting property. 
\begin{lemma}\label{lem:discrete}
The set $E_\tau$ is discrete.
\end{lemma}
\begin{proof}
Let $a\in E_\tau$.
Thanks to Lemma~\ref{lem:interval}, we know that $\lambda^\infty$ is constant on $I_a$ and therefore $a'\not \in E_\tau$ for $a'\in]\inf I_a,a[$. It remains to prove the existence of $\varepsilon>0$ such that $]a,a+\varepsilon[\cap E_\tau=\emptyset$. 

To prove this,
let $m\in \mathbb{N}\setminus\{0\}$ satisfy $\tau\leq 1-\frac{1}{m}$. Choose  $\varepsilon=\frac{1}{2N_a m}$. Let $a'=a+\varepsilon'\in]a,a+\varepsilon[$.  We have $0<\varepsilon'<\varepsilon$. 
For every integer $k$, we have 
$$
kN_aa'=kN_aa+kN_a\varepsilon' \enspace .
$$
As $\varepsilon'\leq \frac{1}{2N_a}$, there exists an integer $k$ such that $kN_a\varepsilon'\leq 1< (k+1)N_a\varepsilon'$. We have then, 
$$ \tau \leq 1-\frac{1}{m}\leq 1-2N_a\varepsilon'< (k-1)N_a\varepsilon'< kN_a\varepsilon'\leq 1.$$
This leads to, adding $(k-1)N_aa$, 
$$
(k-1)N_aa +\tau <
(k-1)N_aa'
<(k-1)N_aa+1.
$$
Since $N_aa\in \mathbb{N}$, this means that $(k-1)N_a a'\in \operatorname{int}(W_\tau)$. If we assume that $a'\in E_\tau$, we have $N_{a'}a'\in\mathbb{N}$ and  then by construction 
$$
\left\{0, a',\dots, N_{a'}a'\right\}\cap W_\tau=\left\{0,N_{a'}a'\right\}
$$
It follows then that
$$
a'\mathbb{N}\cap W_\tau= (N_{a'}a')\mathbb{N},
$$
which in particular implies $a'\mathbb{N}\cap int(W_\tau)=\emptyset$.
This contradicts the existence of $p$.

 We have proved that 
$$
]\inf I_a,a+\varepsilon[\cap E_\tau=\{a\},
$$
which ends the proof of Lemma~\ref{lem:discrete}
\end{proof}

\subsection{\texorpdfstring{Upper bound on $\lambda^{\scaling}$}{Upper bound}}\label{sec:upper_bound}

We next show that $\lambda^\infty$ is an upper bound of $\lambda^\kappa$,
for each $\kappa>0$.  This will turn out to be a relatively simple consequence
of the heuristic argument,
involving the earliest possible times for division,
which led to the formula
of $\lambda^\infty$. 
The proof relies the generational reformulation of the division equation.
Let $a>0,\tau \in]0,1[,\scaling>0$ be fixed. We take the eigenvector as an initial data.

$$
\begin{cases}
\partial_tn_i+\partial_x n_{i} +\scaling\psi(t)\bbbone_{[a,\infty[}(x)n_i(t,x)=0,\\[0.3cm]
n_{i+1}(t,0)=\scaling\psi(t)\int_a^\infty n_i(t,x)dx,\\[0.3cm]
n_0(x)=N^{\scaling}(0,x).
\end{cases}.
$$

We note by $I_{\max}(t)$ the maximal integer $i$ (generation) such that $n_i(t,\cdot)\not\equiv 0$. At time $t=0$, we have $I_{\max}(0)=0$. Since cells need to have an age bigger than $a$ to change generation, we have necessarily $I_{\max}(t)\leq 1$ for $t<a$ and more generally, 
$$
I_{\max}(t)\leq n,\qquad \text{if } t<na.
$$
$$I_{\max}(t)\leq N_a \qquad \text{if } t<(N_a)a.$$
We introduce the notation $p_a=\lceil N_a a\rceil-1$. The integers $N_a,p_a$ are then the minimal integers such that 
$$
p_a+\tau \leq N_aa\leq p_a+1.
$$
We use now the properties of the equation: since no change of generation may occur between times $p_a+\tau$ and $p_a+1$,we have
$$I_{\max}(t)\leq N_a \qquad \text{if } p_a+\tau \leq t<p_a+1.$$
Now, from the definition of $N^{\scaling}$ and $\lambda:=\lambda^{\scaling}$, we have 
$$
\sum_{i=0}^\infty 2^in_i(t,x)=e^{\lambda t}N^{\scaling}(t,x).
$$
In particular, for $t<p_a+1$, we have 
$$
\sum_{i=0}^{I_{\max}(t)} 2^in_i(t,x)=\sum_0^{N_a}2^i n_i(t,x)=e^{\lambda t}N^{\scaling}(t,x).
$$
We integrate with respect to $x$ and obtain 
\begin{equation}\label{eq:upper_1}
e^{\lambda t}\displaystyle\int_0^\infty N^{\scaling}(t,x)dx =\sum_0^{N_a}2^i\displaystyle\int_0^\infty n_i(t,x)dx \leq 2^{N_a}\sum_0^{N_a}\displaystyle\int_0^\infty n_i(t,x)dx.
\end{equation}
But we also have from the properties of the generational formulation and the definition of $I_{\max}$, 
 \begin{equation}\label{eq:upper_2}
 \sum_0^{N_a}\displaystyle\int_0^\infty n_i(t,x)dx=\sum_0^{\infty}\displaystyle\int_0^\infty n_i(t,x)dx=\displaystyle\int_0^\infty n_0(0,x)dx=\displaystyle\int_0^\infty N^{\scaling}(0,x)dx.
 \end{equation}
 Combining (\ref{eq:upper_1}-\ref{eq:upper_2}), we get 
 $$
 e^{\lambda t}\int_0^{\scaling}N^{\scaling}(t,x)dx \leq 2^{N_a}\displaystyle\int_0^\infty N^{\scaling}(0,x)dx, \quad t<p_a+1.
 $$
 As $\scaling$ is fixed, we can use the continuity of $\displaystyle\int_0^\infty N^{\scaling}(t,x)dx$. We have therefore
 $$
\exp\big(\lambda^{\scaling} (p_a+1)\big)\displaystyle\int_0^\infty N^{\scaling}(p_a+1,x)dx\leq 2^{N_a}\displaystyle\int_0^\infty N^{\scaling}(0,x)dx.
 $$
 As $N^{\scaling}(p_a+1,x)=N^{\scaling}(0,x)$, we have
 $$
\exp\big(\lambda^{\scaling} (p_a+1)\big)\leq 2^{N_a}.
 $$
 And therefore, as expected 
 $$
 \lambda^{\scaling}(a)\leq \frac{N_a}{p_a+1}\log 2=\lambda^\infty(a).
 $$\qed

\subsection{\texorpdfstring{Bounding $\lambda$ from below}{Bounding lambda from below}}

\subsubsection*{Step 1: Deriving inequality~\eqref{eq:devil_bounds}} (with $\varepsilon$ instead of $r$.)
We derive the bound through the following technical lemma, in which the sequence
$\theta_1,\ldots,\theta_k$ represents possible successive division times of
an individual of age $1-\tau$ at time $0$. 

\begin{lemma} \label{lem:theta_epsilon}
Suppose there exist $\theta_1,\dots\theta_{N_a}$ and $\varepsilon\geq 0$, satisfying the following properties:
\begin{equation}\label{eq:theta_epsilon}
\bepa
\theta_1\geq \max(a-1+\tau,0),\\
\forall i\leq N_a-1,\quad \theta_{i+1}-(\theta_i+\varepsilon)\geq a,\\
\forall i\in\{1,\dots, N_a\}\qquad \theta_i+\varepsilon<\lfloor \theta_i\rfloor+\tau,
\eepa
\end{equation}
then the following inequality holds:
\beq\label{eq:taux}
e^{\lambda^\scaling(a)(p_a+1)}\geq (2-e^{-\scaling \varepsilon})^{N_a}.
\eeq
\end{lemma}
Intuitively, this lemma
shows that if there is a ``tube'' of positive width $\varepsilon$, limited
from below by
such a sequence $\theta_1,\dots,\theta_{N_a}$
and included in the set of times at
which division is permitted, then, ``enough'' individuals
can follow this tube, so the Floquet eigenvalue $\lambda^\kappa(a)$
can be bounded below by $(N_a\log 2)/(p_a+1)$, up to a correction
of order $O(\exp(-\kappa \epsilon))$ as $\kappa\to\infty$. 

The existence of the sequence $\theta_1,\dots,\theta_{N_a}$ will be established in Step 2 below when $\varepsilon=0$. We shall see in Step~3 that such a sequence does exist, 
for some $\varepsilon>0$, provided that $a$ does not coincide
with the right end $a_r$ of a step of the staircase function $\lambda^\kappa$.
This is illustrated in Figure~\ref{fig-theta}.
\begin{figure}
\begin{center}
\begin{picture}(0,0)%
\includegraphics{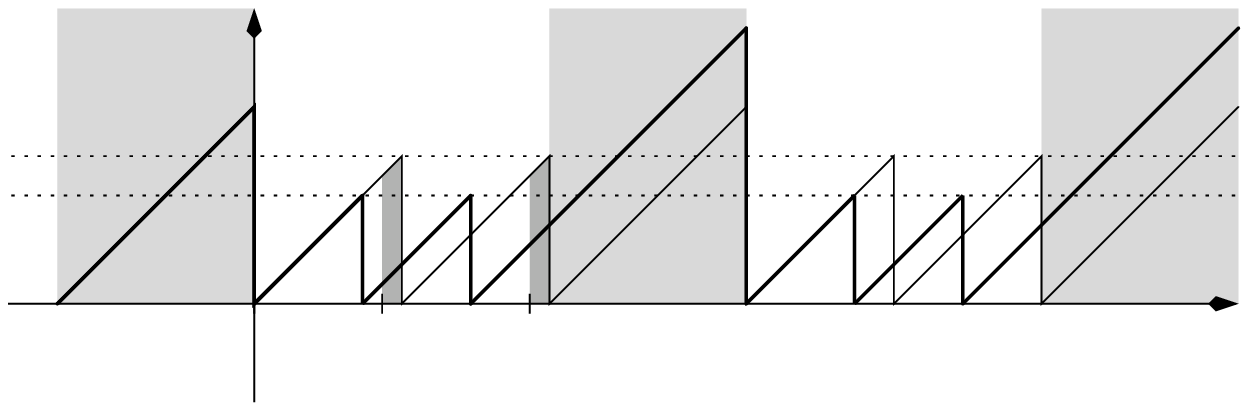}%
\end{picture}%
\setlength{\unitlength}{4144sp}%
\begingroup\makeatletter\ifx\SetFigFont\undefined%
\gdef\SetFigFont#1#2#3#4#5{%
  \reset@font\fontsize{#1}{#2pt}%
  \fontfamily{#3}\fontseries{#4}\fontshape{#5}%
  \selectfont}%
\fi\endgroup%
\begin{picture}(5707,1824)(1741,-5023)
\put(2971,-3436){\makebox(0,0)[lb]{\smash{{\SetFigFont{10}{12.0}{\rmdefault}{\mddefault}{\updefault}{\color[rgb]{0,0,0}age $\alpha(t)$}%
}}}}
\put(7201,-4786){\makebox(0,0)[lb]{\smash{{\SetFigFont{10}{12.0}{\rmdefault}{\mddefault}{\updefault}{\color[rgb]{0,0,0}$t$}%
}}}}
\put(1801,-4786){\makebox(0,0)[lb]{\smash{{\SetFigFont{10}{12.0}{\rmdefault}{\mddefault}{\updefault}{\color[rgb]{0,0,0}$-1+\tau$}%
}}}}
\put(1756,-4201){\makebox(0,0)[lb]{\smash{{\SetFigFont{10}{12.0}{\rmdefault}{\mddefault}{\updefault}{\color[rgb]{0,0,0}$a=0.22$}%
}}}}
\put(1756,-3841){\makebox(0,0)[lb]{\smash{{\SetFigFont{10}{12.0}{\rmdefault}{\mddefault}{\updefault}{\color[rgb]{0,0,0}$a_r=0.3$}%
}}}}
\put(3466,-4786){\makebox(0,0)[lb]{\smash{{\SetFigFont{10}{12.0}{\rmdefault}{\mddefault}{\updefault}{\color[rgb]{0,0,0}$\theta_2$}%
}}}}
\put(4141,-4786){\makebox(0,0)[lb]{\smash{{\SetFigFont{10}{12.0}{\rmdefault}{\mddefault}{\updefault}{\color[rgb]{0,0,0}$\theta_3$}%
}}}}
\put(2971,-4786){\makebox(0,0)[lb]{\smash{{\SetFigFont{10}{12.0}{\rmdefault}{\mddefault}{\updefault}{\color[rgb]{0,0,0}$\theta_1$}%
}}}}
\end{picture}%

\end{center}
\caption{Illustration of Lemma~\ref{lem:theta_epsilon} and of Lemma~\ref{lem:theta_thick}: the trajectory of the jump process $\alpha(t)$ starting from $\alpha(-1+\tau)=0$ is shown in bold for $a=0.22$ and $\tau=0.6$
(division is not permitted in the light ray regions).
 Here, $N_a=3$ and $p_a=0$, so
$\lambda^\infty(a)=3\log 2$. The minimal division age $a$ can be increased
up to $a_r=0.3$, leaving $\lambda^\infty(a)$ unchanged.
The conditions of Lemma~\ref{lem:theta_epsilon}
can be satisfied with a positive $\varepsilon$, because,
for all $a'\in [a,a_r[$, the combinatorial type of the trajectory of the jump process $\alpha(t)$ is unchanged
(trajectory shown by a thin line,  for $a'$ close 
to $a_r$); choosing $\theta_1:=0$, the times $\theta_2$ and $\theta_3$ can be taken to be the absciss\ae\ of the left of the dark grey regions of width $\varepsilon<(a_r-a)/2=0.04$.}
\label{fig-theta}
\end{figure}

\proof

 As for the upper bound, we consider the generational reformulation~\ref{eq:division_generation}, in which $K(t,x)=\scaling\psi(t)\bbbone_{[a,\infty[}(x)$ and the eigenvector $N^{\scaling}(t=0,x)$ is the initial data. 
{From} the definition, we have 
$$
\sum_i 2^in_i(p_a+1,x)=N^{\scaling}(0,x)\exp\big(\lambda^{\scaling}(a)(p_a+1)\big) \enspace .
$$ 
Especially, 

$$
\sum_i 2^i\displaystyle\int_0^\infty n_i(p_a+1,x)dx=\exp\big(\lambda^{\scaling}(a)(p_a+1)\big)\displaystyle\int_0^\infty N^{\scaling}(0,x)dx.
$$ 
We recall that $N^{\scaling}(0,x)=0$ for $x\in [0,1-\tau[$.
As we look for a bound from below, 
thanks to the comparison principle,
it is sufficient to consider the individuals that are the youngest at the beginning, that is the individuals starting at age $(1-\tau)$.
As $K(t,x)$ is nondecreasing with respect to $x$, we know that if we choose $K'(t,x)\leq K(t,x)$ , and $n_i'$ with obvious notations, then we have thanks to the comparison principle
$$
\sum_i 2^i\displaystyle\int_0^\infty n_i'(p_a+1,x)dx\leq \exp\big(\lambda^{\scaling}(a)(p_a+1)\big)\displaystyle\int_0^\infty N^{\scaling}(0,x)dx.
$$ 
We choose 
$$
K'(t,x)=\left\lbrace\begin{array}{l}
\scaling \quad \text{ if } t\in]\theta_1,\theta_1+\varepsilon],\quad x\geq a\\
\scaling \quad \text{ if } t\in]\theta_i+(i-1)\varepsilon,\theta_i+i\varepsilon],\quad x\geq a ,\\
0\quad \text{otherwise}.
\end{array}\right.
$$
{From} the definition of $(\theta_1,\dots \theta_{N_a})$ and $\varepsilon$, we have $K'(t,x)\leq K(t,x)$.
The most important property of the transition rate $K'$ is the following.
\begin{lemma}
Suppose that $n'_i$ are defined by the generational dynamics with transition rate $K'$, then for any $1\leq i\leq N_a$, 
$$
\supp n'_i(\theta_i,\cdot)\subset [a,+\infty[. 
$$
\end{lemma}
\proof We have for $x<a$, $\theta_i-x\not \in \bigcup_i ]\theta_i+(i-1)\varepsilon,\theta_i+i\varepsilon]$  and thereby $K'(\theta_i-x,\cdot)=0$. Therefore for any $j\geq 0$
$$
n'_{j+1}(\theta_i-x,0)=\displaystyle\int_0^\infty K'(\theta_i-x,y)n'_j(\theta_i-x,y)dy=0,
$$ 
and using the characteristics 
$$
n_{j+1}(\theta_i,x)=n_{j+1}(\theta_i-x,0)e^{-\int_0^x K'(\theta_i-x+s,s)ds}=0.
$$
Finally, as $\supp n_0^0\subset [1-\tau,+\infty[$ and as $\supp n_0(t,\cdot)=\supp n_0^0 + t$, we have since $\theta_i\geq a-1+\tau$ for any $i$
$$
\supp n_0(\theta_i,\cdot)\subset [a,+\infty[.
$$
\qed

The latter inclusion means that for any $j>0$, at time $\theta_j$, every individual is mature enough to divide. 
\begin{lemma}\label{lem-induction}
The population $n_i$ satisfies the following equality : defining $I_i(t)=\displaystyle\int_0^\infty n_i(t,x)dx,$ we have 
$$
\forall j>0,\forall i, \quad \left\lbrace\begin{array}{l} I_0(\theta_j+\varepsilon)=e^{-\scaling \varepsilon}I_0(\theta_j),\\
I_{i+1}(\theta_j+\varepsilon)=e^{-\scaling\varepsilon}I_{i+1}(\theta_j)+(1-e^{-\scaling\varepsilon})I_i(\theta_j).
\end{array}\right.
$$
\end{lemma}
\proof It is essentially based on the previous lemma. By construction, we have $\varepsilon <a$. Since we already proved that $\supp \; n_i(\theta_j,\cdot)\subset [a,+\infty[$, we can claim   
$$
\supp \; n_i(\theta_j+\varepsilon,\cdot)\subset [0,\varepsilon]\cup [a+\varepsilon,+\infty[ \;\text{for } i\geq 1 , \;\supp n_0(\theta_j+\varepsilon)\subset [a+\varepsilon,+\infty[
$$  
From the definition of the $n_j$, we have,  
$$
\forall i, \forall x\geq a+\varepsilon,\quad n_i(\theta_j+\varepsilon,x)=n_i(\theta_j,x-\varepsilon)e^{-\int_0^{\varepsilon}K'(\theta_j+s,x-\varepsilon+s)ds}=n_i(\theta_j,x-\varepsilon)e^{-\scaling \varepsilon}.
$$ 
Therefore, 
$$
\int_{x\geq a+\varepsilon} n_i(\theta_j+\varepsilon,x)dx=e^{-\scaling\varepsilon}\int_{x\geq a} n_i(\theta_j,x)dx=e^{-\scaling\varepsilon}I_i(\theta_j).
$$
We have also, for all $i\geq 0$, 
\begin{align*}
\int_{x\leq \varepsilon}n_{i+1}(\theta_j+\varepsilon,x)dx
&=\int_0^\varepsilon n_{i+1}(\theta_j+\varepsilon-x,0)e^{-\int_0^x K'(\theta_j+\varepsilon-x+s,s)ds}dx\\
&=\int_0^\varepsilon n_{i+1}(\theta_j+\varepsilon-x,0)dx\\
&=\int_0^\varepsilon n_{i+1}(\theta_j+s,0)dx.
\end{align*}
Since $\supp n_i(\theta_j,\cdot)\subset [a,\infty[$ and $\varepsilon<a$, we also have 
$$
n_{i+1}(\theta_j+s,0)=\displaystyle\int_0^\infty K'(\theta_j+s,x)n_i(\theta_j+s,x)dx=\int_a^\infty \scaling n_i(\theta_j,x)e^{-\scaling s}=I_i(\theta_j)\scaling e^{\scaling s}.
$$
We integrate and obtain finally
$$
\int_{x\leq \varepsilon}n_{i+1}(\theta_j+\varepsilon,x)dx=(1-e^{-\scaling \varepsilon})I_i(\theta_j).
$$
This completes the proof of the lemma.\qed 

To conclude the proof of Lemma~\ref{lem:theta_epsilon}, we
shall need the following binomial type representation.
\begin{lemma}
Denote by $n\choose p$
the binomial coefficients (with  value $0$ if $p>n$), then we have 
$$
\forall i, I_i(\theta_j+\varepsilon)=I_0(0)
{j\choose i}
(1-e^{\scaling \varepsilon})^i (e^{-\scaling\varepsilon})^{j-i}.
$$
\end{lemma} 
\proof This is readily obtained by induction from Lemma~\ref{lem-induction},
using the fact
that $I_{i}(\theta_j+\epsilon)=I_i(\theta_{j+1})$, for all $j$.
\qed


We finally estimate 
$$
\sum_i 2^i\displaystyle\int_0^\infty n_i'(p_a+1,x)dx =\sum_i 2^i I_i(p_a+1).
$$ 
It follows from the construction of $K'$ that $I_i (p_a+1)=I_i(\theta_{N_a}+\varepsilon)$. Therefore,
\begin{align*}
\sum_i 2^i\displaystyle\int_0^\infty n_i'(p_a+1,x)dx 
&=I_0(0)\sum_i 2^i  {N_a \choose i}
(1-e^{\scaling \varepsilon})^i (e^{-\scaling\varepsilon})^{N_a-i}\\
&= (2(1-e^{-\scaling\varepsilon})+e^{-\scaling\varepsilon})^{N_a}=(2-e^{-\scaling\varepsilon})^{N_a} \enspace .
\end{align*}
This ends the proof of Lemma~\ref{lem:theta_epsilon}.\qed

\subsubsection*{Step 2: Construction of the sequence $\theta_i$, with $\varepsilon=0$}

We need now to prove the existence of the sequence $\theta_i$ used in Lemma~\ref{lem:theta_epsilon}. First we prove the following 
\begin{lemma}\label{lem:theta_0}
There exists a sequence $\theta_i$ as in Lemma~\ref{lem:theta_epsilon} with $\varepsilon=0$.
\end{lemma}
The following $t_k$'s define the dynamics of successive division times of a cell initially of age $x$  changing generation/dividing whenever it reaches age $a$ at a time which is allowed. These will be used afterwards through the construction of $\theta_i$.
\begin{align}
\label{eq:t_i}
t_0(x)=0,\qquad t_{1}(x)=\left\lbrace \begin{array}{l} 
0 \;\text{if } x\geq a,\\
a-x\quad\text{if }a-x\not\in [\tau,1[+\mathbb{N},\\
\lceil a-x\rceil \quad\text{otherwise}\end{array}\right.\\
 t_{i+1}(x)=\left\lbrace \begin{array}{l} t_{i}(x)+a\quad\text{if }t_i(x)\not\in [\tau,1[+\mathbb{N},\\
\lceil t_i(x)+a\rceil \quad\text{otherwise}\end{array}\right.
\nonumber
\end{align}
We now describe the behaviour of the $t_k$'s. First of all, from the definition of $N_a,p_a$, we have the following 
\begin{lemma}
The sequence $t_i$ satisfies the following properties :
\begin{itemize}
\item for $i<N_a$, $t_i(0)=ia$, $t_{N_a}(0)=p_a+1$,
\item for $x>a$, for any $i\geq 0$, $t_{i+1}(x)=t_{i}(0)$,
\item $x\mapsto t_i(x)$ is nonincreasing for any $i$,
\item $\forall i,$ $t_i(x)\leq t_i(0)\leq t_{i+1}(x)$.
\end{itemize} 
\end{lemma}
\proof The first point is an immediate consequence of the construction of $N_a,p_a$. The second point follows from the fact that if $x\geq a$, then $t_1(x)=0$, 
and of the definition of $N_a$. The last two points are straightforward.\qed

An important property of the sequence $t_i$ is given in the 
\begin{lemma}\label{lem:tNa}
Suppose $x> 1-\tau$ or ($x=1-\tau$ and $N_a a\not=p_a+1$), then we have $$t_{N_a}(x)<p_a+1.$$
\end{lemma}
\proof Firstly, this is obvious if $x\geq a$ since in this case $t_{N_a}(x)=t_{N_a-1}(0)$. If $x<a$, we use the following 
\begin{lemma}
Suppose $x<a$, let the $t_i(x)$ be defined by~\eqref{eq:t_i}.
Define 
$$
i_0:=\inf_{t_{i}(x)\in\mathbb{N}} i.
$$
Then we have 
\begin{eqnarray*}
\forall n<N_a,&& t_{i_0+n}(x)=t_{i_0}(x)+na,\\
\forall i<i_0,&& t_i(x)=ia-x,\\
t_{i_0}(x)=\lceil i_0 a-x\rceil &&
\end{eqnarray*}
\end{lemma}
\proof This follows immediately from the definition.\qed

As a consequence of this lemma, we know that if $i_0<N_a$ (taking $n=N_a-i_0$ in the lemma) or if $i_0>N_a$,    we have 
$$
t_{N_a}(x)\not\in\mathbb{N} \;\text{and}\; t_{N_a}(x)\leq t_{N_a}(0)\in\mathbb{N},
$$
therefore, in this case $t_{N_a}(x)<t_{N_a}(0)=p_a+1$. 

If $i_0=N_a$, then, since $x>1-\tau$, 
$$
N_aa-x\leq p_a+1-x<p_a+\tau.
$$
Therefore, we have necessarily $t_{N_a}(x)<p_a+\tau<p_a+1$ which ends the proof of Lemma~\ref{lem:tNa} and thereby of Lemma~\ref{lem:theta_0}. \qed

\subsubsection*{Step 3: Construction of the sequence $\theta_i$, with $\varepsilon>0$.}
Here, we use the notation $a_r=\sup\{a',\lambda^\infty(a')=\lambda^\infty(a)\}$ and $a_l=\inf\{a',\lambda^\infty(a')=\lambda^\infty(a)\}$. 

\begin{lemma}\label{lem:theta_thick}
Suppose there exists a sequence $\theta_1,\dots, \theta_{N_a}$ and $\frac{a_r-a}{2}\geq \varepsilon\geq 0$, such that 
$$
\left\lbrace\begin{array}{l}
\theta_1\geq \max(a-1+\tau,0),\\
\forall i, \; [\theta_i,\theta_i+\varepsilon]\subset [\lfloor\theta_i\rfloor,\lfloor\theta_i\rfloor+\tau],\\
\forall i<N_a,\; \theta_{i+1}\geq \theta_i+\varepsilon+a \enspace.
\end{array}\right.
$$
Then there exists a sequence $\theta'_1,\dots, \theta'_{N_a}$
$$
\left\lbrace\begin{array}{l}
\theta'_1\geq \max(a-1+\tau,0),\\
\forall i, \; [\theta'_i,\theta'_i+\varepsilon']\subset [\lfloor\theta'_i\rfloor,\lfloor\theta'_i\rfloor+\tau],\\
\forall i<N_a,\; \theta'_{i+1}\geq \theta'_i+\varepsilon'+a'.
\end{array}\right.
$$
with $\varepsilon'=\min(\varepsilon+\frac{a-a'}{2},\tau)$.
\end{lemma}

\proof We start the proof with an obvious but necessary remark : if $\lambda^\infty $ is constant on $]a_l,a_r[$, then 
\begin{itemize}
\item $\varepsilon\leq \tau$
\item $a_r-a_l\leq \tau$
\end{itemize}
We build the sequence $\theta_i'$ following this procedure :
$$
\theta_i'=\max(\lfloor \theta_i\rfloor,\theta_i-\frac{a-a'}{2}).
$$
With this construction, we have $\theta'_1\geq a'-1+\tau$ and since $\varepsilon'-\frac{a-a'}{2}\leq \varepsilon$, 
$$
\theta_i'+\varepsilon'\leq \max(\lfloor \theta_i\rfloor+\varepsilon',\theta_i+\varepsilon)\leq \lfloor \theta_i\rfloor+\tau.
$$
Note that we also have $\theta'_i+\varepsilon'\leq \theta_i+\varepsilon+\frac{a-a'}{2}$. This helps us to check the last property:
$$
\theta'_{i+1}-\theta'_i-a'-\varepsilon'\geq \theta_{i+1}-\frac{a-a'}{2}-\theta_i-a'-\varepsilon-\frac{a-a'}{2}=\theta_{i+1}-\theta_i-\varepsilon\geq 0
\enspace .
$$
This ends the proof of the lemma.\qed

To complete the proof of the theorem, we combine the previous arguments. If $\lambda^\infty$ is constant on $[a_l,a_r[$, then for $a<a_r$ we can build a sequence $\theta_i$ as in Lemma~\ref{lem:theta_epsilon} for any $\varepsilon <(a_r-a/2)$ (we first build a sequence for $a_r-0$ and $\varepsilon=0$ and then use Lemma~\ref{lem:theta_thick} to build a sequence with $\varepsilon=a_r-a-0$). Using Lemma~\ref{lem:theta_epsilon}, we get the inequality \eqref{eq:taux} for any $\varepsilon <(a_r-a/2)$. By a standard continuity argument, the same is true for $\varepsilon =(a_r-a/2)=r(a,\tau)$ as defined in \eqref{eq:devil_bounds}. Finally, taking the logarithm of~\eqref{eq:taux} with $\varepsilon=r(a,\tau)$, we get
$$
\lambda^\kappa(a)\geq \lambda^\infty(a)\frac{\log(2-e^{-\kappa r(a,\tau})}{\log 2}.
$$
This leads to~\eqref{eq:devil_bounds} thanks to the following inequality,
which is a consequence of the concavity of the function $\log$:
$$
\forall x\in[0,1], \qquad \log (2-x)\geq (\log 2)(1-x) \enspace .
$$
This ends the proof of the theorem.\qed 

\subsection{Numerical illustrations}\label{sec:num}
We next illustrate Theorem~\ref{Th:devil} by numerical experiments,
that we produced using again monotone finite difference schemes
(see the discussion concerning Figure~\ref{fig:tau05v13510} above).
It should be noted that the bound of the rate of convergence given
in this theorem vanishes when $N_aa$ is a multiple of the period $T$.
This is confirmed by the numerical experiments: the corresponding
values of $a$ are discontinuity points of the staircase function $\lambda^\infty(a)$, and the convergence is seen to be slow at the left of each of these
points.


\begin{figure}[htbp!]
\centering
	\includegraphics[width=0.50\textwidth]{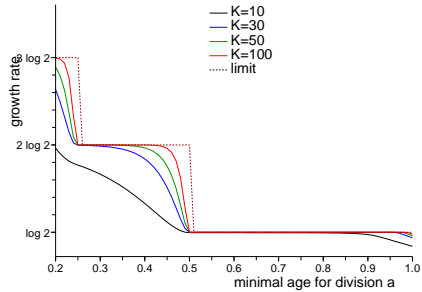}
	\caption{Convergence of the Floquet eigenvalue, when $\tau=1/2$. The curves $a\mapsto \lambda(a)$ are shown for $\kappa=10,30,50,100$. The limit $\lambda^\infty(a)$ is a staircase function.}
	\label{fig:tau05v13510}
\end{figure}
\begin{figure}[htbp!]
	\centering
		\includegraphics[width=0.50\textwidth]{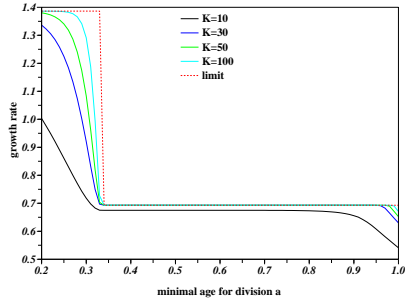}
	\caption{Convergence when $\tau=1/3$. Same conventions as in Figure~\ref{fig:tau05v13510}.}
	\label{fig:figure_devil033}
\end{figure}
%
%
%
%
%
%
\appendix
\section{Appendix: Existence theory for the Floquet eigenvalue}
%
%
%
%
%
%
We shall use the following notation:
\begin{itemize}
\item $\tau_h B=x\mapsto B(x+h)$, extended to take the value zero on $[0,h]$ if $h<0$,
\item $L^\infty_{\text{per}}(0,T,X)$ is the space of bounded $T$-periodic functions taking values in $X$, similarly, $C_{\text{per}}(0,T,X)$ is the space of $T$-periodic continuous (with respect to the time variable) functions taking values in $X$.
\end{itemize}

\begin{theorem}\label{Th:Existence_Floquet}
Assume that $a>0, \scaling>0$,  that $\psi$ is nonnegative, bounded, not identically zero  and that $B$ positive, bounded, satisfying 
\begin{equation}\label{as:integrale_B=infinity}
\forall t, \qquad \displaystyle\int_0^\infty \psi(t-x)B(x)dx=\displaystyle\int_0^\infty \psi(t+x)B(x)dx=+\infty .
\end{equation} 
Then, there exists a unique $\lambda_F>0$ such that there exists $(N,\phi)$ in $ L^{\infty}_{\text{per}}(0,T,L^1(\mathbb{R}^+))\times C_{\text{per}}(0,T,L^\infty(\mathbb{R}^+))$ satisfying (\ref{eq:eigendirect}-\ref{eq:eigendual}-\ref{eq:normalisation}). Furthermore, $\phi$ is unique.
\end{theorem} \vspace{0.5cm}

Before starting the proof, we give a few remarks on the hypotheses:

\begin{itemize}
\item we do not need regularity assumptions on the function $\psi$,
\item when $\min(\psi)>0$, the condition (\ref{as:integrale_B=infinity}) is equivalent to $\displaystyle\int_0^\infty B=+\infty$, in the case $\min(\psi)=0$ it is for instance satisfied if $\min(B)>0$ or at least  $\text{liminf}_{+\infty} B>0$, it is not optimal but other conditions would need assumptions on $a$ and $\scaling$. It could be replaced for instance by 
$$ \forall t, \qquad \int_a^\infty \scaling\psi(t-x)B(x)dx,\int_a^\infty \scaling\psi(t+x)B(x)dx>\log 2,$$
but, as we are studying asymptotic properties, we rather restrict the study to a case where existence does not depend on $a$ nor on $\scaling$,
\end{itemize}

The proof of Theorem~\ref{Th:Existence_Floquet} is based on the method of characteristics and on the Krein-Rutman Theorem, as in~\cite{MMP}, but we need more precisions in order to relax the regularity assumptions on $\psi$. For the sake of simplicity we take $\scaling=1$. We set $P(t):=\int_a^\infty B(x)N(t,x)dx$.  Using the methods of characteristics, we have the following integral equations 
\begin{eqnarray*}
N(t,0)&=& 2\int_a^\infty \psi(t)B(x)N(t-x,0)\exp\big(-\lambda x-\int_a^x \psi(t-x+s)B(s)ds\big)dx\\ &=:&\LL_1^\lambda (N(\cdot,0))(t),\\
P(t)&=& 2\int_a^\infty \psi(t-x)B(x)P(t-x)\exp\big(-\lambda x-\int_a^x \psi(t-x+s)B(s)ds\big)dx\\ &=:&\LL_2^\lambda (P)(t),\\
\phi(t,0)&=& 2\int_a^\infty \psi(t+x)B(x)\phi(t+x,0)\exp\big(-\lambda x-\int_a^x \psi(t+s)B(s)ds\big)dx\\&=:&\LL_3^\lambda (\phi(\cdot,0))(t).
\end{eqnarray*}
We have defined three linear operators on $L^\infty_{per}(0,T)$.
These are well defined as soon as $\lambda>0$. Moreover, we can see  $\LL_2^\lambda$ and $\LL_3^\lambda$ as operators on the space $C_{per}(0,T)$ of T-periodic continuous functions. We have
\begin{lemma}
Under the assumptions of Theorem~\ref{Th:Existence_Floquet}, for any  $\lambda>0$, $\LL_2^\lambda$ and $\LL_3^\lambda$ are nonnegative, compact linear operators on $C_{per}(0,T)$.
\end{lemma}
\myproof The non negativity and linearity are  obvious. We next show the compactness. If we fix $\lambda,a>0$ then for $f$ continuous and $T$-periodic with $\|f\|\leq 1$, if we write $g=\LL_3^\lambda(f)$, we have 
\begin{align*}
g(t+h)&=2\int_{a+h}^\infty \!\!\psi(t+x)B(x-h)f(t+x)\exp\big(\lambda (h-x)
-\int_{a+h}^x\psi(t+s)B(s)ds\big)dx\\
&\!\!\!\!\!\!\!\!\!\!\!\!\!\!\!\!
=2\int_{a}^\infty \!\!\psi(t+x)B(x-h)f(t+x)\exp\big(\lambda(h- x)-\int_{a+h}^x\psi(t+s)B(s)ds\big)dx\\
&\!\!\!\!\!\!\!\!\!\!\!
-2\int_{a}^{a+h} \!\!\psi(t+x)B(x-h)f(t+x)\myexp{\lambda (h- x)-\int_{a+h}^x\psi(t+s)B(s)ds}dx \enspace .
\end{align*}
This leads to 
\begin{align*}
g(t+h)\!-\!g(t)&\!\!= -2\!\int_{a}^{a+h}\!\!\!\!\!\!\!\!\!\psi(t+x)B(x-h)f(t+x)\exp\big(\lambda (h- x)-\!\!\int_{a+h}^x\!\!\!\!\!\!\!\!\psi(t+s)B(s-h)ds\big)dx\\
&\!\!\!\!\!\!\!\!\!\!\!\!\!\!\!\!\!\!\!\!\!\!\!\!\!\!\!\!\!\!\!\!\!\!\!
+ 2(e^{\lambda h}-1)\int_{a}^\infty \psi(t+x)B(x-h)f(t+x)\exp\big(-\!\lambda x-\!\int_{a+h}^x\psi(t+s)B(s-h)ds\big)dx\\
&\!\!\!\!\!\!\!\!\!\!\!\!\!\!\!\!\!\!\!\!\!\!\!\!\!\!\!\!\!\!\!\!\!\!\!+ 2\int_{a}^\infty \psi(t+x)(B(x-h)-B(x))f(t+x)\myexp{-\!\lambda x-\!\int_{a+h}^x\psi(t+s)B(s-h)ds}dx\\
&\!\!\!\!\!\!\!\!\!\!\!\!\!\!\!\!\!\!\!\!\!\!\!\!\!\!\!\!\!\!\!\!\!\!\!+ 2\int_{a}^\infty \psi(t+x)B(x)f(t+x)e^{-\lambda x}\bigg( e^{-\int_{a+h}^x\psi(t+s)B(s-h)ds}-e^{-\int_{a}^x\psi(t+s)B(s)ds}\bigg)dx,\\
&= I_1+I_2+I_3+I_4.
\end{align*}
We deal with each $I_i$ separately,
\begin{equation}
\left|I_1\right|\leq 2\|\psi\|\|B\|e^{\lambda h} h,
\end{equation}
\begin{equation}
\left|I_2\right|\leq 2\|\psi\|\|B\|(e^{\lambda h}-1) \frac{e^{-\lambda a}}{\lambda},
\end{equation}
to make $I_3$ small, we make the following remark: for any $R>a$, 
\begin{equation}
\left|I_3\right|\leq 2\|\psi\|\|\tau_h B-B\|_{L^1([a,R])}+4\|\psi\|\|B\|\int_R^\infty e^{-\lambda x}dx \enspace .\label{e-I3}
\end{equation}
Observe that for all $a$ and $R$, 
$\lim_{h\to 0^+}\|\tau_h B-B\|_{L^1([a,R])}=0$ (indeed, this is obvious
if $B$ is continuous, and, by density of continuous functions in $L^1([a,R])$,
the same is true under the present assumptions).
Since $\lambda$ is positive, the second term in~\eqref{e-I3} can be made
arbitrarily small by choosing $R$ large enough, and then, the first
term can be made arbitrarily small by choosing $h$ small enough. It follows
that for all $\epsilon>0$, we have 
%
\[ \left|I_3\right|\leq \varepsilon \qquad \text{for }h \text{ small enough.}
\]
The same method can be applied to $I_4$. Finally we have the equicontinuity of $\LL_3^\lambda(B)$ where $B$ is the unit ball for the supremum norm. Thanks to Arzela-Ascoli theorem, $\LL_3^\lambda$ is compact.\\
We are now in position to apply Krein-Rutman theorem, there exist $U_2,U_3$ nonnegative eigenvectors of $\LL_2^\lambda,\LL^\lambda_3$ associated to their respective spectral radii $\rho_2(\lambda),\rho_3(\lambda)$.
We have now 
\begin{lemma}
\begin{equation}
\rho_i(\lambda),U_i>0,
\end{equation}
\begin{equation}
\LL_1^\lambda(\psi U_2)=\rho_2(\lambda)\psi U_2,
\end{equation}
\begin{equation}
\rho_3=\rho_2.
\end{equation}
\end{lemma}
\myproof The second statement is a straightforward computation, we prove the first by contradiction: if $U_2$ vanishes, then, for some $t$, 
$$2\int_a^\infty \psi(t-x)B(x)U_2(t-x)\myexp{-\lambda x-\int_a^x \psi(t-x+s)B(s)ds}dx=0,$$
as $B>0$, it would mean thanks to $T$-periodicity $\psi B\equiv 0$ which would lead to $U_2\equiv 0$.
Finally, as $U_2$ does not vanish then $\LL_2^\lambda(U_2)>0$ and therefore, $\rho_2(\lambda)>0$. The equality comes from the duality of operators $\LL_1$ and $\LL_3$, we have
\begin{align*}
\rho_2 \int_0^T \psi(t)U_2(t)U_3(t)dt&=\int_0^T \LL_1(\psi U_2).U_3(t) dt =\int_0^T \psi U_2\LL_3(U_3)dt\\&=\rho_3  \int_0^T \psi(t)U_2(t)U_3(t)dt,
\end{align*}
therefore $\rho_2=\rho_3=\rho$.\\
To end the proof, we need to find $\lambda$ such that $\rho(\lambda)=1$. Obviously, $\rho$ is a decreasing function that vanishes at infinity. 
Since
\begin{align*}
 2\int_a^\infty \!\!\!\psi(t+x)B(x) \myexp{-\!\int_a^x\!\!\psi(t+s)B(s)ds}dx&=2 \Big[\myexp{-\!\int_a^x\!\!\psi(t+s)B(s)ds}\Big]^\infty_a\\ &=2,
\end{align*}
$\rho\rightarrow 2$ as $\lambda\rightarrow 0$. Therefore there exists a unique $\lambda$ satisfying $\rho(\lambda)=1$. Up to a renormalization, $\phi$ and $P$ are unique, and therefore so is $N$. This ends the proof of Theorem~\ref{Th:Existence_Floquet}. 

\section*{Acknowledgement}
The authors thank Jean Clairambault and Benoit Perthame for numerous insightful remarks and discussions. Most of this work has been achieved during the PhD of the second author at Laboratoire Jacques Louis Lions (Paris) and INRIA Rocquencourt.
\bibliographystyle{abbrv}

\begin{thebibliography}{10}

\bibitem{ABG96}
M.~Akian, R.~Bapat, and S.~Gaubert.
\newblock Asymptotics of the {P}erron eigenvalue and eigenvector using max
  algebra.
\newblock {\em C. R. Acad. Sci. Paris.}, 327, S\'erie I:927--932, 1998.

\bibitem{Bacaer}
N.~Bacaer and X.~Abdurahman.
\newblock Resonance of the epidemic threshold in a periodic environment.
\newblock {\em Journal of Mathematical Biology}, 57(5):649--673, November 2008.

\bibitem{Lepoutre_mmnp}
J.~Clairambault, S.~Gaubert, and T.~Lepoutre.
\newblock Comparison of {P}erron and {F}loquet eigenvalues in age structured
  cell division cycle models.
\newblock {\em Math. Model. Nat. Phenom.}, 4(3):183--209, 2009.
\newblock Also eprint arXiv:0812.0803.

\bibitem{Lepoutre_MCM}
J.~Clairambault, S.~Gaubert, and T.~Lepoutre.
\newblock Circadian rhythm and cell population growth.
\newblock {\em Mathematical and Computer Modelling}, 53(7-8):1558 -- 1567,
  2011.
\newblock Mathematical Methods and Modelling of Biophysical Phenomena.

\bibitem{CGP}
J.~Clairambault, S.~Gaubert, and B.~Perthame.
\newblock An inequality for the {P}erron and {F}loquet eigenvalues of monotone
  differential systems and age structured equations.
\newblock {\em C. R. Math. Acad. Sci. Paris}, 345(10):549--554, 2007.

\bibitem{CMP}
J.~Clairambault, P.~Michel, and B.~Perthame.
\newblock Circadian rhythm and tumour growth.
\newblock {\em C. R. Acad. Sci.}, 342(1):17--22, 2006.

\bibitem{CMP2}
J.~Clairambault, P.~Michel, and B.~Perthame.
\newblock A mathematical model of the cell cycle and its circadian control.
\newblock In D.~I.A., B.~L., B.~H., de~Vries~G., and H.~H. P., editors, {\em
  Mathematical modeling of Biological Systems}, pages 247--259. Birkh{\"a}user,
  2007.

\bibitem{Dibrov}
B.~F. Dibrov, A.~M. Zhabotinsky, Y.~A. Neyfakh, M.~P. Orlova, and L.~I.
  Churikova.
\newblock Mathematical model of cancer chemotherapy. {P}eriodic schedules of
  phase-specific cytotoxic-agent administration increasing the selectivity of
  therapy.
\newblock {\em Math. Biosci.}, 73(1):1--31, 1985.

\bibitem{Diekmann_periodic}
O.~Diekmann, H.~J. A.~M. Heijmans, and H.~R. Thieme.
\newblock On the stability of the cell-size distribution. {II}. {T}ime-periodic
  developmental rates.
\newblock {\em Comput. Math. Appl. Part A}, 12(4-5):491--512, 1986.
\newblock Hyperbolic partial differential equations, III.

\bibitem{Doumic}
M.~Doumic.
\newblock Analysis of a population model structured by the cells molecular
  content.
\newblock {\em Mathematical Modelling of Natural Phenomena}, 2(3):121--152,
  2007.

\bibitem{Goldbeterbook}
A.~Goldbeter.
\newblock {\em Biochemical oscillations and cellular rhythms}.
\newblock Cambridge University Press, 1997.

\bibitem{McKendrick}
W.~O. Kermack and A.~McKendrick.
\newblock {A Contribution to the Mathematical Theory of Epidemics}.
\newblock {\em Proceedings of the Royal Society of London. Series A, Containing
  Papers of a Mathematical and Physical Character}, 115(772):700--721, Aug.
  1927.

\bibitem{Levi}
F.~Levi and U.~Schibler.
\newblock Circadian rhythms: mechanisms and therapeutic implications.
\newblock {\em Annu Rev Pharmacol Toxicol}, 47:593--628, 2007.

\bibitem{Magal_periodic}
P.~Magal.
\newblock Compact attractors for time-periodic age-structured population
  models.
\newblock {\em Electron. J. Differential Equations}, pages No. 65, 35 pp.
  (electronic), 2001.

\bibitem{marshall79}
A.~Marshall and I.~Olkin.
\newblock {\em Inequalities: Theory of majorization and its applications}.
\newblock Academic Press, New York, 1979.

\bibitem{MetzDiekmann}
J.~Metz and O.~Diekmann.
\newblock {\em The dynamics of physiologically structured populations},
  volume~68 of {\em L.N. in biomathematics}.
\newblock Springer, 1986.

\bibitem{MMP}
P.~Michel, S.~Mischler, and B.~Perthame.
\newblock General relative entropy inequality: an illustration on growth
  models.
\newblock {\em J. Math. Pures et Appl.}, 84(9):1235--1260, May 11 2005.

\bibitem{PerthameBook}
B.~Perthame.
\newblock {\em Transport equations in biology}.
\newblock Birkh{\"a}user, 2007.

\bibitem{Thieme_84}
H.~R. Thieme.
\newblock Renewal theorems for linear periodic {V}olterra integral equations.
\newblock {\em J. Integral Equations}, 7(3):253--277, 1984.

\bibitem{thieullen}
P.~Thieullen and E.~Garibaldi.
\newblock Description of some ground states by puiseux techniques.
\newblock {\em J. Stat. Phys.}, 146:125--180, 2012.

\bibitem{Webb_book}
G.~F. Webb.
\newblock {\em Theory of nonlinear age-dependent population dynamics},
  volume~89 of {\em Monographs and Textbooks in Pure and Applied Mathematics}.
\newblock Marcel Dekker Inc., New York, 1985.

\end{thebibliography}

\end{document}